%% file: thesis.tex
\documentclass[leqno]{report}
\usepackage{amsmath}
\usepackage{amsthm}
\usepackage{amssymb}
\usepackage{amsfonts}
\usepackage{amsxtra}     
\usepackage{rac}                        
\usepackage{dbl12}
\usepackage{srcltx}
\usepackage{epsfig}
\usepackage{verbatim}
\usepackage[all,arc]{xy}
\usepackage[all]{xy}
\usepackage{textcomp}
\usepackage{cite}

\theoremstyle{plain}
\newtheorem{theorem}{Theorem}

\newtheorem{corollary}[theorem]{Corollary}
\newtheorem{lemma}[theorem]{Lemma}

\theoremstyle{definition}
\newtheorem{definition}[theorem]{Definition}
\newcommand\xqed[1]{%
  \leavevmode\unskip\penalty9999 \hbox{}\nobreak\hfill
  \quad\hbox{#1}}
\newcommand\demo{\xqed{$\Diamond$}}
\newtheorem{notation}[theorem]{Notation}

\newtheorem{example}[theorem]{Example}

\theoremstyle{remark}

\include{header}

\numberwithin{theorem}{chapter}        

\begin{document}

\bibliographystyle{plain}       


\titlepage{Classifying Topoi and Preservation of Higher Order Logic by Geometric Morphisms}{Shawn J. Henry}{Doctor of Philosophy}
{Mathematics} {2013}
{ Professor Andreas R. Blass, Chair \\
  Professor Mel Hochster\\
  Assistant Professor Scott Schneider\\
  Professor G. Peter Scott \\
  Associate Professor James P. Tappenden }


\initializefrontsections

\copyrightpage{Shawn J. Henry}

\setcounter{page}{1}


\startacknowledgementspage

{I am most indebted to my advisor, Andreas Blass, for his patience and for sharing a small part of his immense knowledge with me.  I would also like to thank Fran\c{c}ois Dorais for encouraging me not to give up when I realized that topos theory is hard.}

\tableofcontents

\startthechapters 

\chapter{Introduction}
\label{chap1}
\input{intro}
\chapter{Internal Geometric Theories and Classifying Topoi}
\label{chap2} 
\input{chap1}

\chapter{Infinitary Deductions and Inductive Constructions}
\label{chap3}
\input{chap2}
\chapter{Splitting Epimorphisms}
\label{chap4}
\input{chap3}
\chapter{Geometric Dedekind Finiteness}
\label{chap5}
\input{chap4}

\chapter{Field Objects}
\label{chap6}
\input{chap5}


\bibliography{biblio}   

\end{document}

%% file: intro.tex
Elementary topoi are categories that share many of the properties of the familiar category of sets and functions.  As such, they can serve as universes for the interpretation of higher order logic in much the same way that classical models of higher order theories are constructed from sets, functions, and relations.  Morphisms of topoi come in two flavors: logical functors which preserve all of the topos structure in sight, and therefore preserve all interpretations of higher order logic, and geometric morphisms, which abstract the morphisms of topoi of sheaves on topological spaces induced by continuous functions.  In general, geometric morphisms need not preserve power objects, and therefore need not preserve interpretations of all higher order theories; some theories, however, are nonetheless preserved.  Among the known examples are models of finitary algebraic theories, K-finite objects, well-founded relations, inductive constructions, and models of existential fixed point theories.   
	
One technique which has been particularly useful for demonstrating the preservation of certain higher order theories, those which admit certain related internal first order geometric theories of “bad sets” which intuitively prevent a structure from satisfying the theory in question, makes use of classifying topoi.  Roughly speaking, a classifying topos for an internal geometric theory $\mathbb{T}$ in a topos $\mathcal{E}$ is a topos $\mathcal{E}[\mathbb{T}]$ defined over $\mathcal{E}$ with a universal model of $\mathbb{T}$ such that every model of $\mathbb{T}$ in every topos $\mathcal{F}$ defined over $\mathcal{E}$ is the pullback of the universal model along an essentially unique geometric morphism from $\mathcal{F}$ to $\mathcal{E}[\mathbb{T}]$.  If the classifying topos $\mathcal{E}[\mathbb{T}]$ is degenerate, $\mathbb{T}$ has models in a topos $\mathcal{F}$ defined over $\mathcal{E}$ if and only if $\mathcal{F}$ is degenerate, and one can often show that the classifying topos for the theory of bad sets for a higher order theory $\mathcal{T}$ is degenerate if and only if $\mathcal{T}$ is true in every structure in question.
	
Having developed the general theory of degeneracy of classifying topoi, we apply it to higher order theories which admit a natural notion of bad set.  An object $A$ in a topos is said to be Dedekind finite if every monomorphism $m:A\rightarrow A$ is an isomorphism.  A bad set for $A$ is the graph of a monomorphism which misses a point. In general, if $A$ is Dedekind finite, its pullback along a geometric morphism need not be; however, if we impose the condition that the classifying topos for bad sets is degenerate, then the pullback of $A$ along any geometric morphism will also be Dedekind finite.  The theory of such objects is the internalization in the higher order logic of topoi of the external notion of geometric finiteness introduced by Freyd in \cite{Freyd06}.   
	
A  non-example  arises in the theory of field objects in topoi.  The degeneracy of the classifying topos for bad sets for a certain version of the theory of fields (which is classically, but not intuitionistically equivalent to the usual one) is equivalent to to satisfying the theory, but this equivalence does not imply that its models are preserved by geometric morphisms.\bigskip  

The following is a summary of the contents of each of the chapters.\bigskip

\textbf{II.)  Internal Geometric Theories and Classifying Topoi}	

Notation is fixed and the notions under consideration are defined.  The construction of classifying topoi for internal geometric theories (indeed, internal geometric theories themselves) appears to be largely folklore.  Most of the treatments of classifying topoi give the construction explicitly for geometric theories in Grothendieck topoi, then remark that, with some care, the same construction can be carried out for internal geometric theories in any elementary topos.  The usual construction of the classifying topos makes use of the full syntax of geometric logic to define the syntactic site of a theory, and a major difficulty arises from the fact that it is highly non-trivial to formalize the full deduction calculus for geometric logic in an arbitrary elementary topos.  In this chapter we present a syntax-free construction of the classifying topos for geometric propositional and full geometric theories, following \cite{BlaSc83}, where a similar construction is given for universal Horn theories.\bigskip

\textbf{III.)  Infinitary Deductions and Inductive Constructions}
  
In this chapter we introduce a limited deduction calculus for geometric logic, called the Grothendieck deduction calculus, which is modeled on the closure axioms for Grothendieck topologies. We introduce inductive constructions as a means of internally formalizing proofs in the Grothendieck deduction calculus.  We show that the Grothendieck deduction calculus proves a contradiction from a geometric theory $\mathbb{T}$ if and only if the classifying topos of $\mathbb{T}$ is degenerate.  Thus the Grothendieck deduction calculus is necessary and sufficient for determining when the classifying topos of a geometric theory is degenerate.  Inductive constructions are proof-theoretic by nature.  We address the question of whether every inductive construction arises in proving that the classifying topos of a geometric theory is degenerate.  We show that this is true for a limited class of inductive constructions, namely the ones which we will call “downward stratified”.\bigskip

\textbf{IV.)  Splitting Epimorphisms}

In this chapter we use classifying topoi to show that for every epimorphism $e:A\rightarrow B$ in a Boolean topos $\mathcal{E}$, there is a geometric surjection $f:\mathcal{B}\rightarrow \mathcal{E}$ from a Boolean topos $\mathcal{B}$ such that (internally) $f^{*}(e)$ has a splitting in $\mathcal{B}$.  We remark that the same construction splits epimorphisms over an arbitrary topos $\mathcal{E}$, but the geometric morphism $f$ is not a surjection in general.  We will make several uses of this result in the following chapter.\bigskip

\textbf{V.)  Geometric Dedekind Finiteness}

In this chapter we present an example of the technique developed in the second chapter.  We discuss Dedekind finite objects in topoi, which we show are not preserved by geometric morphisms.  We introduce Geometrically Dedekind Finite (GDF) objects as those objects for which the classifying topos for bad sets for Dedekind finiteness is degenerate.  We study the properties of GDF objects and give several examples.\bigskip  

\textbf{VI.)  Field Objects}

In this chapter we present a non-example of the technique developed in the second chapter.  We discuss field objects in topoi, and we show that one of the notions of field object is not preserved by geometric morphisms, but, nonetheless, is equivalent to the degeneracy of the classifying topos of its theory of bad sets.  We resolve the apparent paradox. 

%% file: chap1.tex
\section{Topoi and Geometric Morphisms}
\label{Topoi and Geometric Morphisms}

In this section we will define the basic notions under consideration, namely elementary topoi and geometric morphisms, and fix notation.  

\begin{definition}
\cite{Johns77}  An \textit{elementary topos} is a category $\mathcal{E}$ such that:

1.)  $\mathcal{E}$ has all finite limits, that is:

\begin{itemize}

\item[] i.)  $\mathcal{E}$ has a terminal object $1$ 
 
\item[] ii.)  $\mathcal{E}$ has a pullback for every diagram $X\rightarrow A \leftarrow Y$ 

\end{itemize}

2.)  $\mathcal{E}$ is cartesian closed, that is, for any pair of objects $A,B$ there is an object $B^{A}$ and for any object $X$ an isomorphism Hom$_{\mathcal{E}}(X\times A,B)\cong$ Hom$_{\mathcal{E}}(X,B^{A})$ natural in $X$.

3.)  $\mathcal{E}$ has a subobject classifier, that is, an object $\Omega$ and a morphism $t:1\rightarrow\Omega$ such that for any monomorphism $m:B\rightarrow A$ there is a unique morphism $\chi_{m}:A\rightarrow\Omega$ such that the diagram:

\[
\xymatrix{B  \ar[d]^-{m}  \ar[r]  &  1 \ar[d]^-{t}\\
A  \ar[r]^-{\chi_{m}}  &  \Omega} 
\]

is a pullback. \demo

\end{definition}

\notation

Topoi will usually be denoted by script letters $\mathcal{E},\mathcal{F},\mathcal{G}$, etc. and their objects by capital italic letters $A,B,C$, etc.  We shall sometimes refer to the object $\Omega^{A}$, called the \textit{power object} of $A$, as $P(A)$ for convenience. 

We will also require that our topoi have a natural numbers object:

\begin{definition}
 
A \textit{natural numbers object} in a topos $\mathcal{E}$ is an object $\mathbb{N}$ with morphisms $0:1\rightarrow \mathbb{N}$ and $s:\mathbb{N}\rightarrow \mathbb{N}$ such that for any object $A$ and morphisms $a:1\rightarrow A$ and $m:A\rightarrow A$ there is a unique morphism $f:\mathbb{N}\rightarrow A$ such that $f\circ 0=a$ and $f\circ s=m\circ f$. \demo

\end{definition}

The general properties of topoi with a natural numbers object can be found in \cite{Johns77}.  In particular topoi also have all finite colimits.  

\begin{definition}
 
A topos $\mathcal{B}$ is called \textit{Boolean} if every monomorphism $m:B\rightarrow A$ has a complement, that is, there is a monomorphism $\neg m:\neg B\rightarrow A$ such that $B\cup\neg B\cong A$ and $B\cap\neg B\cong 0$, where $\cup$ and $\cap$ mean union and intersection of subobjects, respectively.  \demo 

\end{definition}

As we mentioned in the introduction, there are two kinds of morphisms between topoi.  We shall be concerned with the following:

\begin{definition}

A \textit{geometric morphism} $f:\mathcal{F}\rightarrow\mathcal{E}$ between topoi $\mathcal{F}$ and $\mathcal{E}$ consists of a pair of adjoint functors $f_{*}:\mathcal{F}\rightarrow\mathcal{E}$ and $f^{*}:\mathcal{E}\rightarrow\mathcal{F}$ such that the left adjoint $f^{*}$ is left exact, that is, preserves finite limits. \demo

\end{definition}

For a more detailed exposition, see \cite{Johns77}.  Topoi are in many ways similar to the category of sets and functions.  In particular, they possess an internal logic which allows us to interpret higher order logic and argue much the same way as in ordinary mathematics based on sets, provided that we don't make use of the law of the excluded middle (unless we are working in a Boolean topos), the full axiom of replacement, or the axiom of choice.  For this reason we will sometimes refer to objects in a topos as sets.  A detailed exposition of the internal logic can be found in \cite[Ch.4]{Bell88}.  Our arguments, unless otherwise specified, take place in (can be formalized in) the Kripke-Joyal semantics as formulated there, without explicit mention of generalized elements or the forcing relation.  When our arguments involve transfer along the inverse image part of a geometric morphism $f:\mathcal{F}\rightarrow\mathcal{E}$ and $a\in A$ refers to a generalized element $a:X\rightarrow A$ of an object $A$ of $\mathcal{E}$, $f^{*}(a)$ refers to the generalized element $f^{*}(a):f^{*}(X)\rightarrow f^{*}(A)$ of $f^{*}(A)$ in $\mathcal{F}$.  

\section{Higher Order Logic}
\label{The Internal Logic}

In this section we describe higher order logic interpretable in the internal logic of topoi and give examples that will be used throughout the thesis.  

\begin{definition}

A \textit{signature} $\Sigma$ for higher order logic consists of:

1.)  A set $\Sigma$-Sort of \textit{sorts}

2.)  A set $\Sigma$-Fun of \textit{function symbols} 

3.)  A set $\Sigma$-Rel of \textit{relation symbols} \demo

\end{definition}

\begin{definition}
 
The set $\Sigma$-Type of \textit{types} over $\Sigma$ is defined recursively as follows:

1.)  \textit{Basic Types}: Each sort is a type

2.)  \textit{Product Types}: There is a distiguished type $1$, and if $A$ and $B$ are types, so is $A\times B$

3.)  \textit{Function Types}:  If $A$ and $B$ are types, so is $[A\rightarrow B]$

4.)  \textit{Power Types}:  If $A$ is a type, so is $P(A)$

5.)  \textit{List Types}:  If $A$ is a type, so is $L(A)$\demo 

To each function symbol $F$ we assign a pair $(A,B)$ of types and write $F:A\rightarrow B$ to indicate that $F$ has the pair of types $(A,B)$, and to each relation symbol $R$ we assign a type $A$ and write $R\subseteq A$ to indicate that $R$ has type $A$.

\end{definition}

Each type comes with a stock of variables, and terms, formulas, and sentences are defined recursively from these.  We write $a\epsilon A$ to indicate that $a$ is a variable of type $A$.  A \textit{sequent} is a sentence of the form $\forall x_{0}...\forall x_{k}(\varphi\rightarrow\psi)$ where $\varphi$ and $\psi$ are formulas with at most the variables $x_{0},...,x_{k}$ free.  For details see \cite[D4]{Johns02}.  Topoi have exactly the structure necessary to interpret higher order logic.  

\begin{definition}
 
Let $\Sigma$ be a higher order signature.  A $\Sigma$-\textit{structure} $\mathcal{M}$ in a topos $\mathcal{E}$ is given by specifying for each type $A$ an object $A^{\mathcal{M}}$ of $\mathcal{E}$ subject to the requirements that:

1.)  $1^{\mathcal{M}}$ is the terminal object of $\mathcal{E}$ and $(A\times B)^{\mathcal{M}}=A^{\mathcal{M}}\times B^{\mathcal{M}}$ where the latter is the product in $\mathcal{E}$

2.)  $[A\rightarrow B]^{\mathcal{M}}$ is the exponential ${B^{\mathcal{M}}}^{A^{\mathcal{M}}}$

3.)  $P(A)^{\mathcal{M}}=P(A^{\mathcal{M}})$ where the latter is the power object in $\mathcal{E}$

4.)  $L(A)^{\mathcal{M}}=(A^{\mathcal{M}})^{[n]}$ where the latter is the list object as defined below.

And we specify:

1.)  For each function symbol $F$ of type $(A,B)$, a morphism $F^{\mathcal{M}}:A^{\mathcal{M}}\rightarrow B^{\mathcal{M}}$

2.)  For each relation symbol $R$ of type $A$, a subobject $R^{\mathcal{M}}\subseteq A^{\mathcal{M}}$ \demo  

\end{definition}

The interpretation of types and symbols can be extended to an interpretation of terms and formulas over $\Sigma$ in the context of a suitable set of variables.  See \cite[D4]{Johns02} for details.  In particular, if $\varphi$ and $\psi$ are formulas with only variables among $x_{0},...,x_{k}$ of types $A_{0},...,A_{k}$ free, then $\varphi^{\mathcal{M}}$ and $\psi^{\mathcal{M}}$ are interpreted as subobjects of $A_{0}\times...\times A_{k}$ in the context $x_{0},...,x_{k}$ and we say the sequent $\forall x_{0}...\forall x_{k}(\varphi\rightarrow\psi)$ is true in $\mathcal{M}$ if $\varphi^{\mathcal{M}}\subseteq\psi^{\mathcal{M}}$.  See \cite[D4]{Johns02} for details.  

If $f:\mathcal{F}\rightarrow\mathcal{E}$ is a geometric morphism, pulling the interpretations of the sorts and symbols back along $f^{*}$ gives us a $\Sigma$-structure $f^{*}(\mathcal{M})$ in $\mathcal{F}$.  We say a formula $\varphi$ is (strongly) preserved by $f$ if for every $\Sigma$-structure $\mathcal{M}$ in $\mathcal{E}$ we have $f^{*}(\varphi^{\mathcal{M}})=\varphi^{f^{*}(\mathcal{M})}$.  We say that a sequent $\forall x_{0}...\forall x_{k}(\varphi\rightarrow\psi)$ is (weakly) preserved by $f$ if whenever it is true in a $\Sigma$-structure $\mathcal{M}$ in $\mathcal{E}$ it is also true when interpreted in $f^{*}(\mathcal{M})$.

We're now ready to give some examples of constructions defined in the higher order logic of topoi.  We will present the constructions informally, but one could, if one wished, write down a signature and formulas that describe them.

\begin{example}

Let $A$ be an object in a topos $\mathcal{E}$ and let $\varphi(X)$ be the formula $\forall a\epsilon A,\forall p\epsilon P(A)((\emptyset\in X)\wedge(p\in X\rightarrow p\cup\{a\}\in X)$ and define $K(A)=\{p\in P(A)|\forall X\in PP(A)(\varphi(X)\rightarrow p\in X))\}$, that is, $K(A)$ consists of all those subobjects of $A$ which belong to every family that contains the empty set and is closed under adjoining singletons.  A subobject $p$ of $A$ that belongs to $K(A)$ is called a \textit{K-finite} subobject of $A$.  $A$ itself is called K-finite if $A\in K(A)$.  K-finiteness is preserved by geometric morphisms.  In fact, for any object $A$ in a topos $\mathcal{E}$ and geometric morphism $f:\mathcal{F}\rightarrow\mathcal{E}$ we have $f^{*}(K(A))\cong K(f^{*}(A))$.  For this and other properties of K-finite objects, see \cite[D5.4]{Johns77}.  

\end{example}

\begin{example}

Let $\mathcal{E}$ be a topos with natural numbers object $\mathbb{N}$.  In the slice topos $\mathcal{E}/\mathbb{N}$ we have the generic natural number $n$ given by the diagonal $\delta:\mathbb{N}\rightarrow \mathbb{N}\times\mathbb{N}$.  Like any natural number, it has a cardinal $[n]=\{k\in\mathbb{N}|k<n\}$.  If $A$ is any object of $\mathcal{E}$ we can form the exponential $(\mathbb{N}^{*}A)^{[n]}$ in $\mathcal{E}/\mathbb{N}$.  We denote its domain in $\mathcal{E}$ by $A^{[n]}$.  The object $A^{[n]}$ is called the \textit{list object} over $A$.  Intuitively, $A^{[n]}$ is the object of finite tuples of elements of $A$, in the sense that, for any natural number $p:1\rightarrow \mathbb{N}$, the object $A^{[p]}$ of $p$-tuples of elements of $A$ is, up to isomorphism, the pullback of $A^{[n]}\rightarrow \mathbb{N}$ along $p$.  This construction is preserved by inverse images, since slicing commutes with inverse images and $\mathbb{N}$ and exponentiation by finite cardinals are preserved by them. For more details see \cite[A2.5.15]{Johns02}.\bigskip

\end{example}

\begin{example}

Let $\Sigma$ be a signature with a single type $A$ and a single relation symbol of arity $(A,A)$.  A $\Sigma$-structure is called an \textit{internal partial order} if it satisfies: 

1.)  $\forall a\epsilon A(a\leq a)$

2.)  $\forall a, b\epsilon A(a\leq b\wedge b\leq a\rightarrow a=b)$

3.)  $\forall a,b,c\epsilon A(a\leq b\wedge b\leq c\rightarrow a\leq c)$

If, in addition, $\Sigma$ has binary relation symbols and $\vee$ and $\wedge$ of input arity $(A,A)$ and output arity $A$, a $\Sigma$-structure is called a \textit{distributive lattice} if it is an internal partial order and it satisfies:

1.)  $\forall a,b,c\epsilon A(a\leq a\vee b$ and $b\leq a\vee b$ and if $a\leq c$ and $b\leq c$ then $a\vee b\leq c)$

2.)  $\forall a,b,c\epsilon A(a\wedge b\leq a$ and $a\wedge b\leq b$ and if $c\leq a$ and $c\leq b$ then $c\leq a\wedge b)$

3.)  $\forall a,b,c\epsilon A(a\vee(b\wedge c)=(a\vee b)\wedge (a\vee c)$ and $a\wedge (b\vee c)=(a\wedge b)\vee (a\wedge c))$

If, in addition, $\Sigma$ has 0-ary function symbols $0$ and $1$ of type $A$, a $\Sigma$-structure is called a \textit{bounded distributive lattice} if it is a distributive lattice and it satisfies:

1.)  $\forall a\epsilon A(a\leq 1)$

2.)  $\forall a\epsilon A(0\leq a)$

If, in addition, $\Sigma$ has a binary function symbol $\Rightarrow$ of input arity $(A,A)$ and output arity $A$, a $\Sigma$-structure is called an \textit{internal Heyting algebra} if it is a bounded distributive lattice and it satisfies:

1.)  $\forall a,b,c\epsilon A(c\wedge a\leq b$ iff $c\leq a\Rightarrow b)$

If, in addition, $\Sigma$ has a unary function symbol $\bigvee$ of input type $P(A)$ and output type $A$, a $\Sigma$-structure is called a \textit{internal complete Heyting algebra} if it is a Heyting algebra and it satisfies:

1.)  $\forall b\epsilon A\forall B\epsilon P(A)(b\in B\rightarrow b\leq\bigvee B)$

2.)  $\forall a\epsilon A\forall B\epsilon P(A)(\forall b\epsilon A(b\in B\rightarrow b\leq a)\rightarrow\bigvee B\leq a)$

3.)  $\forall a\epsilon A\forall B\epsilon P(A)(a\wedge\bigvee B=\bigvee\{a\wedge b|b\in B\})$

The subobject classifier $\Omega$ of a topos $\mathcal{E}$ always has the structure of an internal complete Heyting algebra.  In particular, the subobject classifier has internal K-finite meets, and arbitrary joins, and these commute, making $\Omega$ into an internal \textit{frame}.  See \cite[A5.11]{Johns02} for a proof.

\end{example}

\section{Internal Propositional Theories}
\label{Internal Propositional Theories}

In this section we will define a fragment of the internal logic of topoi which is analogous to propositional logic in classical mathematics.  We will primarily be making use of geometric propositional theories rather than full geometric theories (to be defined in the next section), so we explicitly describe them here.

\begin{definition}
 
Let $\mathcal{E}$ be a topos.  An \textit{internal geometric propositional language} $\mathcal{L}$ in $\mathcal{E}$ consists of an object $A$ of $\mathcal{E}$ thought of as the ``set of sentence symbols" and the object $Seq_{A}=K(A)\times P(K(A))$ thought of as the ``set of implications between finite conjunctions and arbitrary disjunctions of finite conjunctions of sentence symbols".  An \textit{internal geometric propositional theory} $\mathbb{T}$ over $\mathcal{L}$ is a subobject of $K(A)\times P(K(A))$.  We will often drop ``geometric'', as we shall not be considering propositional theories that are not geometric.  \demo

\end{definition}

\notation 

An element $(p,\alpha)\in Seq_{A}$ is to be thought of as the sequent
\[
\xymatrix{a_{0}\wedge...\wedge a_{n}\rightarrow \bigvee_{i\in I}a_{0}^{i}\wedge...\wedge a_{n_{i}}^{i} }
\]
where $p=\{a_{0},...,a_{n}\}$ and $\alpha = \{p_{i}|i\in I\}$ and $p_{i}=\{a_{0}^{i},...,a_{n_{i}}^{i}\}$.  We shall use the suggestive notation $p\rightarrow \bigvee_{q\in\alpha}q$ for the sequent $(p,\alpha)$ and we shall write $\top$ in place of $\emptyset\in K(A)$ and $\bot$ in place of $\emptyset\in P(K(A))$. \demo

If $\mathbb{T}$ is an internal propostional theory in a topos $\mathcal{E}$ and $f:\mathcal{F}\rightarrow\mathcal{E}$ is a geometric morphism, there is a natural translation $f^{*}(\mathbb{T})$ of $\mathbb{T}$ in $\mathcal{F}$.  We have $f^{*}(K(A))\cong K(f^{*}(A))$ and there is a canonical morphism $f^{*}(P(K(A)))\rightarrow P(K(f^{*}(A)))$ given by the transpose of the characteristic map of the monomorphism $f^{*}(\varepsilon_{K(A)})\rightarrow f^{*}(P(K(A)))\times K(f^{*}(A))$ where $\varepsilon_{K(A)}$ is the subobject of $P(K(A))\times K(A)$ classified by the evaluation map.  If $T$ denotes the underlying object of $\mathbb{T}$ we define $f^{*}(\mathbb{T})$ to be the image of $f^{*}(T)$ under the map $f^{*}(K(A)\times P(K(A)))\rightarrow K(f^{*}(A))\times P(K(f^{*}(A)))$.   

\begin{definition}

A \textit{model} of an internal propositional theory $\mathbb{T}$ in a topos $\mathcal{E}$ is a subobject $X\rightarrow A$ such that when we extend $\chi_{X}$ to $K(A)\times P(K(A))$ by defining $\chi_{X}(p)=\chi_{X}(\{a_{0},...,a_{k}\})=\bigwedge_{i=0}^{k}\chi_{X}(a_{i})$ and $\chi_{X}(\bigvee_{q\in\alpha}q)=\bigvee_{q\in\alpha}\chi_{X}(q)$, we have $\chi_{X}(p)\leq\chi_{X}(\bigvee_{q\in\alpha}q)$ for every sequent $p\rightarrow\bigvee_{q\in\alpha}q$ of $\mathbb{T}$, where the latter operations are interpreted in the internal complete Heyting algebra structure of $\Omega$.  Notice that, given a sequent $p\rightarrow\bigvee_{q\in\alpha}q$ we may assume that $p\subseteq q$ for all $q\in\alpha$ since, as one easily checks, $p\rightarrow\bigvee_{q\in\alpha}p\cup q$ is true in a structure if and only if $p\rightarrow\bigvee_{q\in\alpha}q$ is.   \demo

\end{definition}

\begin{lemma}

If $X\rightarrow A$ is a model of an internal propositional theory $\mathbb{T}$ in a topos $\mathcal{E}$ and $f:\mathcal{F}\rightarrow\mathcal{E}$ is a geometric morphism, then $f^{*}(X)\rightarrow f^{*}(A)$ is a model of $f^{*}(\mathbb{T})$ in $\mathcal{F}$. 

\end{lemma}

\begin{proof}

Let $\Omega_{\mathcal{E}}$ and $\Omega_{\mathcal{F}}$ denote the subobject classifiers of $\mathcal{E}$ and $\mathcal{F}$.  The characteristic map of $f^{*}(X)$ is the composition of the canonical map $f^{*}(\Omega_{\mathcal{E}})\rightarrow\Omega_{\mathcal{F}}$, the characteristic map of $f^{*}(t_{\mathcal{E}})$, with $f^{*}(\chi_{X})$.  Since $f^{*}(K(A))$ is canonically isomorphic to $K(f^{*}(A))$, the extension of $\chi_{f^{*}(X)}$ to $K(f^{*}(A))$ is just $f^{*}$ applied to the extension of $\chi_{X}$ to $K(A)$ composed with the canonical map.  Now if $\alpha\in P(K(A))$ and $p\in\alpha$ then $f^{*}(p)$ is in the image of $f^{*}(\alpha)$ under the canonical map $f^{*}(\Omega_{\mathcal{E}}^{K(A)})\rightarrow\Omega_{\mathcal{F}}^{K(f^{*}(A))}$, so the extension of $\chi_{f^{*}(X)}$ to $P(K(f^{*}(A)))$ at $f^{*}(\alpha)$ is greater than or equal to the composition of the canonical morphism $f^{*}(\Omega_{\mathcal{E}})\rightarrow\Omega_{\mathcal{F}}$ with $f^{*}$ applied to the extension of $\chi_{X}$ to $P(K(A))$.  The result follows immediately.  For more details, see \cite[D1]{Johns02}. 

\end{proof}

\section{Classifying Topoi for Internal Propositional Theories}
\label{Classifying Topoi for Internal Propositional Theories}

Like any geometric theory, internal propositional theories have classifying topoi.  If $\mathbb{T}$ is an internal propositional theory in a topos $\mathcal{E}$, the underlying category of the site for the classifying topos is the object $K(A)$ considered as a poset with the opposite of the usual inclusion order.  A sieve on a K-finite set $p\in K(A)$ is a subobject $R$ of $K(A)$ such that $\forall q\in R(p\subseteq q)$ and if $q\in R$ and $q\subseteq r$ then $r\in R$.  A Grothendieck topology on $K(A)$ is an assignment $J$ to each $p\in K(A)$ of a family of sieves on $p$ such that:

1.)  The maximal sieve $\{q\in K(A)|p\subseteq q\}$ is in $J(p)$

2.)  If $R\in J(p)$ and $p\subseteq p'$ then the pullback $\{q\in R|p'\subseteq q\}$ of $R$ to $p'$ is in $J(p')$

3.)  If $R\in J(p)$ and $R'$ is a sieve on $p$ such that the pullback of $R'$ to every $q\in R$ is in $J(q)$ then $R'\in J(p)$

We are now ready to construct the classifying topos for $\mathbb{T}$.  To each axiom $p\rightarrow\bigvee_{q\in\alpha}q$ of $\mathbb{T}$, where we may assume that $\forall q\in\alpha(p\subseteq q)$, we associate a sieve $R_{(p,\alpha)}=\{r\in K(A)|\exists q\in\alpha(q\subseteq r)\}$ on $K(A)$.  

If $R$ is a subset of $K(A)$ we say that $R$ \textit{directly covers} $r\in K(A)$ if either:

1.)  The maximal sieve on $p$ is contained in $R$, or

2.)  There is an axiom $p\rightarrow\bigvee_{q\in\alpha}q$ of $\mathbb{T}$ such that $p\subseteq r$ and $\forall q\in\alpha(r\cup q\in R)$

The notion of direct covering, when restricted to sieves, gives the pullback closure of the collection of sieves associated to the axioms of $\mathbb{T}$, as well as the maximal sieves.  To obtain the closure under the third requirement for Grothendieck topologies, we define for each subset $R$ of $K(A)$ the subset $\bar{R}$ to be the smallest subset of $K(A)$ such that $R\subseteq\bar{R}$ and $\bar{R}$ contains everything it directly covers.  Then the Grothendieck topology generated by the axioms of $\mathbb{T}$ is $J(p)=\{R\subseteq K(A)|R$ is a sieve on $p$ and $p\in\bar{R}\}$.  This construction/terminology is from \cite{Blass87}.  To prove our main theorem of this section we will need the following:

\begin{definition}
 
Let $\mathbb{C}$ be an internal category in a topos $\mathcal{E}$ with internal Grothendieck topology $J$ (the extension of the usual definition of a category and Grothendieck topology in the internal logic of $\mathcal{E}$) and let $f:\mathcal{F}\rightarrow\mathcal{E}$ be a geometric morphism.  A functor $F:\mathbb{C}\rightarrow\mathcal{F}$ is an internal functor $F:f^{*}(\mathbb{C})\rightarrow\mathcal{F}$ in $\mathcal{F}$.  Such a functor $F$ is called \textit{flat} if for each object $A$ of $\mathcal{F}$ the internal comma category $(A\downarrow F)$ is cofiltered, and \textit{continuous} if for each internal covering sieve $R\in J(c)$ the internal family $F(R)$ of morphisms to $F(c)$ is jointly epimorphic.  See \cite{Diaco75}. \demo

\end{definition}

\begin{lemma}

Let $\mathbb{C}$ and $F$ be as above.  If $\mathbb{C}$ has all finite limits then $F$ is flat if and only if it preserves them. 

\end{lemma}

\begin{proof}
 
See \cite{Diaco75}.

\end{proof}

\begin{theorem}
 
(Diaconescu)  Let $\mathbb{C}$ be an internal category in a topos $\mathcal{E}$ with internal Grothendieck topology $J$.  There is an equivalence,natural in $\mathcal{F}$, of the category $\mathfrak{Top}/\mathcal{E}(\mathcal{F},Sh_{J}(\mathcal{E}^{\mathbb{C}^{op}}))$ of geometric morphisms over $\mathcal{E}$ and natural transformations of their left adjoint parts, and the category \textbf{ContFlat}$(\mathbb{C},\mathcal{F})$ of continuous flat functors and natural transformations.

\end{theorem}

\begin{proof}
 
See \cite{Diaco75}.

\end{proof}

We're now ready to prove our main result of this section.

\begin{theorem}

$Sh_{J}(\mathcal{E}^{{K(A)}^{op}})$ is the classifying topos for $\mathbb{T}$.

\end{theorem}

\begin{proof}
 
By Diaconescu's theorem, geometric morphisms $f:\mathcal{F}\rightarrow Sh_{J}(\mathcal{E}^{{K(A)}^{op}})$ correspond to internal continuous flat functors $F:f^{*}(K(A))\rightarrow\mathcal{F}$.  Working in the internal logic of $\mathcal{F}$ we show that such internal continuous flat functors correspond to models of $\mathbb{T}$ in $\mathcal{F}$.  We suppress mention of $f^{*}$ since the argument takes place entirely in $\mathcal{F}$.

$(\Rightarrow)$  Let $F$ be a continuous flat functor on $K(A)$.  Since $K(A)$ has all finite limits and $F$ preserves them, $F$ takes values in $P(1)=\Omega$.  The composite of $F$ with the inclusion $A\rightarrow K(A)$ defines a subobject $X$ of $A$ and by flatness, $F$ agrees with the extension of the characteristic map of $X$ to $K(A)$.  Since $F$ is continuous it turns joins into epimorphic families, hence, if $p\rightarrow\bigvee_{q\in\alpha}q$ is an axiom of $\mathbb{T}$ we must have $F(p)\leq \bigvee_{q\in\alpha}F(q)$, but this is exactly the requirement that $X$ be a model of $\mathbb{T}$.

$(\Leftarrow)$  Let $X\rightarrow A$ be a model of $\mathbb{T}$ in $\mathcal{F}$.  Let $F$ be the extension of $\chi_{X}$ to $K(A)$.  $F$ is evidently a functor and preserves finite limits, and the condition that $X$ is a model of $\mathbb{T}$ is precisely the condition that $F$ turns covers associated to the axioms into epimorphic families.  But the pullback of an epimorphic family is again an epimorphic family, so $F$ must turn all covers into epimorphic families, so $F$ is a continuous flat functor.

These two constructions are evidently mutually inverse. 

\end{proof}

Since $K(A)$ has a terminal object, the empty set, the classifying topos of an internal propositional theory on $A$ is degenerate precisely when the empty set is covered by the empty sieve, or, more specifically, if every subset of $K(A)$ that contains everything it directly covers contains the empty set. 

\section{Internal Geometric Theories}
\label{Internal Geometric Theories}

\begin{definition}

An \textit{internal language} $\mathcal{L}$ in a topos $\mathcal{E}$ consists of:

1.)  An object $S$ of sorts

2.)  An object $C$ of constant symbols together with a morphism $\tau_{C}:C\rightarrow S$ assigning to each constant symbol a sort

3.)  An object $Rel$ of relation symbols together with a morphism $\tau_{Rel}:Rel\rightarrow S^{[n]}$ assigning to each relation symbol a tuple of sorts (sometimes called its ``arity")

4.)  An object $V\cong\mathbb{N}$ of variables\demo

\end{definition}

Since $V$ has decidable equality we can form the subobject $Dist\subseteq V^{[n]}$ consisting of tuples of distinct variables; that is, $Dist=\{(x_{0},...,x_{k})\in V^{[n]}|\bigwedge_{i<j}(x_{i}\neq x_{j})\}$.

\begin{definition}
 
Let $\mathcal{L}$ be an internal language in a topos $\mathcal{E}$.  The object $Con$ of \textit{contexts} for $\mathcal{L}$ is the subobject of $Dist\times S^{[n]}$ consisting of pairs $(\overrightarrow{x},\overrightarrow{s})$ where $\overrightarrow{x}$ and $\overrightarrow{s}$ have the same length.  $Con$ is equipped with a morphism $\tau_{Con}:Con\rightarrow S^{[n]}$ given by the second projection.  We think of a context as assiging a sort to each variable in the tuple $\overrightarrow{x}$.  We will usually supress mention of the list of sorts and just write $\overrightarrow{x}$ when talking about a context.\demo

\end{definition}

We now define formulas-in-context in stages:

\begin{definition}
 
1.)  The object $VTerm_{\mathcal{L}}$ of \textit{variable terms-in-context} is the subobject of $Con\times V$ consisting of pairs $(\overrightarrow{x},x)$ where $x$ is among the variables $\overrightarrow{x}$.  The object $CTerm_{\mathcal{L}}$ of \textit{constant terms-in-context} is the subobject of $Con\times V\times C$ consisting of triples $(\overrightarrow{x},x,c)$ where $x$ is a variable among the list $\overrightarrow{x}$ and $c$ is a constant symbol with the same sort as $x$.  The object $Term_{\mathcal{L}}$ of \textit{terms-in-context} is $VTerm_{\mathcal{L}}\sqcup CTerm_{\mathcal{L}}$.  We denote a term in the context $\overrightarrow{x}$ by $\overrightarrow{x}:t$.  $Term_{\mathcal{L}}$ is a equipped with a morphism $\tau_{Term_{\mathcal{L}}}:Term_{\mathcal{L}}\rightarrow S$ where $\tau_{Term_{\mathcal{L}}}(\overrightarrow{x}:t)$ is the $i$th sort of $\tau_{Con}(\overrightarrow{x})$ if $t$ is the variable $x_{i}$ or the sort $\tau_{C}(t)$ if $t$ is a constant symbol.

2.)  The object $RAtom_{\mathcal{L}}$ of \textit{relational atomic formulas-in-context} is the subobject of $(Con\times Rel\times Term_{\mathcal{L}}^{[n]})$ consisting of triples $(\overrightarrow{x},R,(\overrightarrow{x}:t_{0},...,\overrightarrow{x}:t_{k}))$ where the $\overrightarrow{x}:t_{i}$ are terms all in the context $\overrightarrow{x}$ and $R$ is a relation symbol with $\tau_{Rel}(R)=(\tau_{Term_{\mathcal{L}}}(\overrightarrow{x}:t_{0}),...,\tau_{Term_{\mathcal{L}}}(\overrightarrow{x}:t_{k}))$.  The object $EqAtom_{\mathcal{L}}$ of \textit{equational atomic formulas-in-context} is the subobject of $Con\times Term_{\mathcal{L}}\times Term_{\mathcal{L}}$ consisting of triples $(\overrightarrow{x},\overrightarrow{x}:t,\overrightarrow{x}:t')$  where $\overrightarrow{x}:t$ and $\overrightarrow{x}:t'$ are terms in the context $\overrightarrow{x}$ and $\tau_{Term_{\mathcal{L}}}(\overrightarrow{x}:t)=\tau_{Term_{\mathcal{L}}}(\overrightarrow{x}:t')$.  We denote a relational atomic formula in the context $\overrightarrow{x}$ by $\overrightarrow{x}:R(\overrightarrow{t})$ and we denote an equational atomic formula in the context $\overrightarrow{x}$ by $\overrightarrow{x}:(t=t')$.  The object $Atom_{\mathcal{L}}$ of \textit{atomic formulas-in-context} is $RAtom_{\mathcal{L}}\sqcup EqAtom_{\mathcal{L}}$. 

3.)  The object $Horn_{\mathcal{L}}$ of \textit{Horn formulas-in-context} (this name is non-standard, see\cite[D1.1.3]{Johns02}) is the subobject of $Con\times K(Atom_{\mathcal{L}})$ consisting of pairs $(\overrightarrow{x},\{\overrightarrow{x}:\alpha_{0},...,\overrightarrow{x}:\alpha_{k}\})$ where each $\overrightarrow{x}:\alpha_{i}$ is an atomic formula in the context $\overrightarrow{x}$.  We denote a Horn formula in the context $\overrightarrow{x}$ by $\overrightarrow{x}:\alpha_{0}\wedge...\wedge\alpha_{k}$.

4.)  The object $Reg_{\mathcal{L}}$ of \textit{regular formulas-in-context} is the subobject of $Con\times Con\times Horn_{\mathcal{L}}$ consisting of triples $(\overrightarrow{x},\overrightarrow{y},\phi)$ where $\phi$ is a Horn formula in the context $\overrightarrow{x}\overrightarrow{y}$.  We denote a regular formula in the context $\overrightarrow{x}$ by $\overrightarrow{x}:\exists\overrightarrow{y}\phi$.

5.)  The object $Geom_{\mathcal{L}}$ of \textit{geometric formulas-in-context} is the subobject of $Con\times P(Reg_{\mathcal{L}})$ consisting of pairs $(\overrightarrow{x},\Gamma)$ where $\Gamma$ is a set of regular formulas all in the context $\overrightarrow{x}$.  We denote a geometric formula in the context $\overrightarrow{x}$ by $\overrightarrow{x}:\bigvee_{\gamma\in\Gamma}\gamma$\demo 

\end{definition}

\begin{definition}
 
Let $\mathcal{L}$ be an internal language in a topos $\mathcal{E}$.  The object $GeomSeq_{\mathcal{L}}$ of \textit{geometric sequents} is the subobject of $Con\times Horn_{\mathcal{L}}\times Geom_{\mathcal{L}}$ consisting of triples $(\overrightarrow{x},\phi,\Gamma)$ such that $\phi$ and $\Gamma$ both have context $\overrightarrow{x}$.  We denote a geometric sequent by $\phi\rightarrow\bigvee_{\gamma\in\Gamma}\gamma$ and think of a geometric sequent as an implication between a horn formula and a geometric formula in the same context, universally quantified over the variables in the context.

An \textit{internal geometric theory} over $\mathcal{L}$ is a subobject $\mathbb{T}$ of $GeomSeq_{\mathcal{L}}$.\demo

\end{definition}

If $\mathcal{L}$ is an internal language and $f:\mathcal{F}\rightarrow\mathcal{E}$ is a geometric morphism, $f^{*}(\mathcal{L})$ denotes the internal langage $(f^{*}(S),f^{*}(C),f^{*}(\tau_{C}),f^{*}(Rel),f^{*}(\tau_{Rel}),f^{*}(V))$ in $\mathcal{F}$.  (Recall that the construction $A^{[n]}$ and $\mathbb{N}$ are preserved.)  Notice that the construction of $Term_{\mathcal{L}}$, $Atom_{\mathcal{L}}$, $Horn_{\mathcal{L}}$, and $Reg_{\mathcal{L}}$ are preserved by $f$; that is, for example, $f^{*}(Reg_{\mathcal{L}})\cong Reg_{f^{*}(\mathcal{L})}$.  The canonical morphism $f^{*}(P(Reg_{\mathcal{L}}))\rightarrow P(f^{*}(Reg_{\mathcal{L}}))\cong P(Reg_{f^{*}(\mathcal{L})})$ gives us a morphism $f^{*}(Geom_{\mathcal{L}})\rightarrow Geom_{f^{*}(\mathcal{L})}$ and therefore a morphism $trans:f^{*}(GeomSeq_{\mathcal{L}})\rightarrow GeomSeq_{f^{*}(\mathcal{L})}$.  Given an internal geometric theory $\mathbb{T}$ over $\mathcal{L}$ in $\mathcal{E}$ we obtain an internal geometric theory $\mathbb{T}_{f^{*}(\mathcal{L})}$, namely the image of $f^{*}(\mathbb{T})$ in $GeomSeq_{f^{*}(\mathcal{L})}$ under $trans$.  Once we have defined what an $\mathcal{L}$-structure is in $\mathcal{E}$ we will be able to define an $\mathcal{L}$-structure in any topos with a geometric morphism to $\mathcal{E}$, namely a structure for $f^{*}(\mathcal{L})$.  Once we have defined what it means for an $\mathcal{L}$-structure to be a model of an internal geometric theory $\mathbb{T}$ we will be able to say what it means for an $\mathcal{L}$-structure to be a model of $\mathbb{T}$ in any topos with a geometric morphism to $\mathcal{E}$, namely that it be a model of $\mathbb{T}_{f^{*}(\mathcal{L})}$.

\begin{definition}
 
Let $\mathcal{L}$ be an internal language in a topos $\mathcal{E}$.  An $\mathcal{L}$-\textit{structure} $\mathcal{M}$ consists of:

1.)  An object $M$ and a morphism $\sigma:M\rightarrow S$ assigning a sort to each element of $M$

2.)  A morphism $\kappa:C\rightarrow M$ such that $\sigma\circ\kappa=\tau_{C}$ assigning to each constant symbol an element of $M$ of the same sort.

3.)  A subobject $Rel^{\mathcal{M}}$ of the pullback of $\tau_{Rel}$ along $\sigma^{[n]}:M^{[n]}\rightarrow S^{[n]}$ assigning to each relation symbol a set of tuples of $M$ of the same sort\demo 

\end{definition}

We are now ready to explain how to interpret a formula-in-context in an $\mathcal{L}$-structure.
 
Let $\mathcal{M}$ be an $\mathcal{L}$-structure in a topos $\mathcal{E}$.  

1.)  Let $P_{Term_{\mathcal{L}}}$ denote the pullback of $\sigma^{[n]}$ along the morphism $\tau_{Con}\circ\pi_{1}:Term_{\mathcal{L}}\rightarrow S^{[n]}$ that sends a term $\overrightarrow{x}:t$ to $\tau_{Con}(\overrightarrow{x})$.  Let \textlbrackdbl\textrbrackdbl$_{Term_{\mathcal{L}}}:P_{Term_{\mathcal{L}}}\rightarrow M$ be the morphism such that \textlbrackdbl$(\overrightarrow{x}:t,\overrightarrow{m})$\textrbrackdbl$_{Term_{\mathcal{L}}}$ is $m_{i}$ when $t=x_{i}$ or $t=(x_{i},c)$.  We call \textlbrackdbl\textrbrackdbl$_{Term_{\mathcal{L}}}$ the \textit{interpretation} of $Term_{\mathcal{L}}$ in $\mathcal{M}$. 

2.)  Let $P_{RAtom_{\mathcal{L}}}$ denote the pullback of $\sigma^{[n]}$ along the morphism $\tau_{Con}\circ\pi_{1}:RAtom_{\mathcal{L}}\rightarrow S^{[n]}$ that sends a relational atomic formula $\overrightarrow{x}:R(\overrightarrow{t})$ to $\tau_{Con}(\overrightarrow{x})$.  Let $Proj_{R}:P_{RAtom_{\mathcal{L}}}\rightarrow P_{Term_{\mathcal{L}}}^{[n]}$ be the morphism that sends $(\overrightarrow{x}:R(\overrightarrow{t}),\overrightarrow{m})$ to $((\overrightarrow{x}:t_{0},\overrightarrow{m}),...,(\overrightarrow{x}:t_{k},\overrightarrow{m}))$ and let $Q_{RAtom_{\mathcal{L}}}$ denote the pullback of $(1_{Rel}\times$\textlbrackdbl\textrbrackdbl$_{Term_{\mathcal{L}}}^{[n]})\circ (pi_{2}\times Proj_{R}):P_{RAtom_{\mathcal{L}}}\rightarrow Rel\times M^{[n]}$ along the inclusion $Rel^{\mathcal{M}}\subseteq Rel\times M^{[n]}$.  The inclusion \textlbrackdbl\textrbrackdbl$_{RAtom_{\mathcal{L}}}:Q_{RAtom_{\mathcal{L}}}\subseteq P_{RAtom_{\mathcal{L}}}$ is the interpretation of $RAtom_{\mathcal{L}}$ in $\mathcal{M}$.  

Similarly, let $P_{EqAtom_{\mathcal{L}}}$ denote the pullback of $\sigma^{[n]}$ along the morphism $\tau_{Con}\circ \pi_{1}:EqAtom_{\mathcal{L}}\rightarrow S^{[n]}$.  Let $Proj_{Eq}:P_{EqAtom_{\mathcal{L}}}\rightarrow P_{Term_{\mathcal{L}}}\times P_{Term_{\mathcal{L}}}$ be the morphism that sends $(\overrightarrow{x}:(t=t'),\overrightarrow{m})$ to $((\overrightarrow{x}:t,\overrightarrow{m}),(\overrightarrow{x}:t',\overrightarrow{m}))$ and let $Q_{EqAtom_{\mathcal{L}}}$ denote the pullback of \textlbrackdbl\textrbrackdbl$_{Term_{\mathcal{L}}}\times$\textlbrackdbl\textrbrackdbl$_{Term_{\mathcal{L}}}\circ Proj_{Eq}:P_{EqAtom_{\mathcal{L}}}\rightarrow M\times M$ along the diagonal $M\rightarrow M\times M$.  The inclusion \textlbrackdbl\textrbrackdbl$_{EqAtom_{\mathcal{L}}}:Q_{EqAtom_{\mathcal{L}}}\subseteq P_{EqAtom_{\mathcal{L}}}$ is the interpretation of $EqAtom_{\mathcal{L}}$ in $\mathcal{M}$.   

The morphism \textlbrackdbl\textrbrackdbl$_{Atom_{\mathcal{L}}}=$\textlbrackdbl\textrbrackdbl$_{RAtom_{\mathcal{L}}}\sqcup$\textlbrackdbl\textrbrackdbl$_{EqAtom_{\mathcal{L}}}:QAtom_{\mathcal{L}}=Q_{RAtom_{\mathcal{L}}}\sqcup Q_{EqAtom{L}}\rightarrow P_{Atom_{\mathcal{L}}}=P_{RAtom_{\mathcal{L}}}\sqcup P_{EqAtom_{\mathcal{L}}}$ is the interpretation of $Atom_{\mathcal{L}}$ in $\mathcal{M}$

3.)  Let $p_{Atom_{\mathcal{L}}}:Atom_{\mathcal{L}}\rightarrow P(M^{[n]})$ and $q_{Atom_{\mathcal{L}}}:Atom_{\mathcal{L}}\rightarrow P(M^{[n]})$ be the transposes of the classifying maps for $P_{Atom_{\mathcal{L}}}$ and $Q_{Atom_{\mathcal{L}}}$ respectively as subobjects of $Atom_{\mathcal{L}}\times M^{[n]}$.  Let $P_{Horn_{\mathcal{L}}}$ and $Q_{Horn_{\mathcal{L}}}$ denote the subobjects classified by the respective transposes of the maps $\bigcap\circ\exists p_{Atom_{\mathcal{L}}}\circ \pi_{1}:Horn_{\mathcal{L}}\rightarrow P(M^{[n]})$ and $\bigcap\circ\exists q_{Atom_{\mathcal{L}}}\circ \pi_{1}:Horn_{\mathcal{L}}\rightarrow M^{[n]}$ where $\bigcap$ is the internal intersection map and $\exists p_{Atom_{\mathcal{L}}}$ and $\exists q_{Atom_{\mathcal{L}}}$ are the respective image maps of $p_{Atom_{\mathcal{L}}}$ and $q_{Atom_{\mathcal{L}}}$.  Since $Q_{Atom_{\mathcal{L}}}\subseteq P_{Atom_{\mathcal{L}}}$ we have $Q_{Horn_{\mathcal{L}}}\subseteq P_{Horn_{\mathcal{L}}}$ and one easily checks that $P_{Horn_{\mathcal{L}}}$ is isomorphic to the pullback of $\sigma^{[n]}$ along the morphism $\tau_{Con}\circ \pi_{1}:Horn_{\mathcal{L}}\rightarrow S^{[n]}$.  The morphism \textlbrackdbl\textrbrackdbl$_{Horn_{\mathcal{L}}}:Q_{Horn_{\mathcal{L}}}\subseteq P_{Horn_{\mathcal{L}}}$ is the interpretation of $Horn_{\mathcal{L}}$ in $\mathcal{M}$.

4.)  Let $P_{Reg_{\mathcal{L}}}$ denote the pullback of $\sigma^{[n]}$ along the morphism $\tau_{Con}\circ \pi_{1}:Reg_{\mathcal{L}}\rightarrow S^{[n]}$ and let $P'_{Reg_{\mathcal{L}}}$ denote the pullback of $\sigma^{[n]}$ along the morphism $\tau_{Con}\circ \pi_{12}:Reg_{\mathcal{L}}\rightarrow S^{[n]}$.  Let $Proj_{Reg_{\mathcal{L}}}:P'_{Reg_{\mathcal{L}}}\rightarrow P_{Reg_{\mathcal{L}}}$ be the map which sends $((\overrightarrow{x},\overrightarrow{y},\phi),\overrightarrow{m}\overrightarrow{n})$ to $((\overrightarrow{x},\overrightarrow{y},\phi),\overrightarrow{m})$ and let $Proj_{Horn_{\mathcal{L}}}:P'_{Reg_{\mathcal{L}}}\rightarrow P_{Horn_{\mathcal{L}}}$ be the map that sends $((\overrightarrow{x},\overrightarrow{y},\phi),\overrightarrow{m}\overrightarrow{n})$ to $((\overrightarrow{x}\overrightarrow{y},\phi),\overrightarrow{m}\overrightarrow{n})$.  Let $Q'_{Reg_{Mathcal{L}}}$ denote the pullback of $Proj_{Horn_{\mathcal{L}}}$ along the inclusion \textlbrackdbl\textrbrackdbl$_{Horn_{\mathcal{L}}}$ and let $Q_{Reg_{\mathcal{L}}}$ be the image of $Q'_{Reg_{\mathcal{L}}}$ in $P_{Reg_{\mathcal{L}}}$ under $Proj_{Reg_{\mathcal{L}}}$.  The inclusion \textlbrackdbl\textrbrackdbl$_{Reg_{\mathcal{L}}}:Q_{Reg_{\mathcal{L}}}\rightarrow P_{Reg_{\mathcal{L}}}$ is the interpretation of $Reg_{\mathcal{L}}$ in $\mathcal{M}$.

5.)  Let $p_{Geom_{\mathcal{L}}}:Reg_{\mathcal{L}}\rightarrow P(M^{[n]})$ and $q_{Reg_{\mathcal{L}}}:Reg_{\mathcal{L}}\rightarrow P(M^{[n]})$ be the transposes of the classifying maps for $P_{Reg_{\mathcal{L}}}$ and $Q_{Reg_{\mathcal{L}}}$ respectively as subobjects of $Reg_{\mathcal{L}}\times M^{[n]}$.  Let $P_{Geom_{\mathcal{L}}}$ and $Q_{Geom_{\mathcal{L}}}$ denote the subobjects classified by the respective transposes of the maps $\bigcup\circ\exists p_{Reg_{\mathcal{L}}}\circ \pi_{1}:Reg_{\mathcal{L}}\rightarrow P(M^{[n]})$ and $\bigcup\circ\exists q_{Reg_{\mathcal{L}}}\circ \pi_{1}:Reg_{\mathcal{L}}\rightarrow P(M^{[n]})$ where $\bigcup$ is the internal union map and $\exists p_{Reg_{\mathcal{L}}}$ and $\exists q_{Reg_{\mathcal{L}}}$ are the respective image maps of $p_{Atom_{\mathcal{L}}}$ and $q_{Atom_{\mathcal{L}}}$.  Since $Q_{Reg_{\mathcal{L}}}\subseteq P_{Reg_{\mathcal{L}}}$ we have $Q_{Geom_{\mathcal{L}}}\subseteq P_{Geom_{\mathcal{L}}}$ and and one easily checks that $P_{Geom_{\mathcal{L}}}$ is isomorphic to the pullback of $\sigma^{[n]}$ along the morphism $\tau_{Con}\circ \pi_{1}:Geom_{\mathcal{L}}\rightarrow S^{[n]}$.  The morphism \textlbrackdbl\textrbrackdbl$_{Geom_{\mathcal{L}}}:Q_{Geom_{\mathcal{L}}}\subseteq P_{Geom_{\mathcal{L}}}$ is the interpretation of $Geom_{\mathcal{L}}$ in $\mathcal{M}$.

\begin{definition}
 
Let $\mathcal{M}$ be an $\mathcal{L}$-structure, let $\mathbb{T}$ be an internal geometric theory, and let $P_{\mathbb{T}}$ denote the pullback of $(\sigma^{[n]},\sigma^{[n]}):M^{[n]}\rightarrow S^{[n]}\times S^{[n]}$ along the morphism $(\tau_{Con}\circ\pi_{1},\tau_{Con}\circ\pi_{1}):\mathbb{T}\rightarrow S^{[n]}\times S^{[n]}$.   $P_{\mathbb{T}}$ has obvious projections to $P_{Horn_{\mathcal{L}}}$ and $P_{Geom_{\mathcal{L}}}$.  We say that $\mathcal{M}$ is a \textit{model} of $\mathbb{T}$ if the pullback of $Q_{Horn_{\mathcal{L}}}$ along the projection to $P_{Horn_{\mathcal{L}}}$ is a subobject of the pullback of $Q_{Geom_{\mathcal{L}}}$ along the projection to $P_{Geom_{\mathcal{L}}}$.   

\end{definition}

\begin{definition}

Let $\mathcal{M}$ and $\mathcal{N}$ be $\mathcal{L}$-structures.  An $\mathcal{L}$-\textit{structure homomorphism} is a morphism $h:M\rightarrow N$ such that:

1.)  $\sigma_{\mathcal{N}}\circ h=\sigma_{\mathcal{M}}$

2.)  $\tau_{C,\mathcal{N}}\circ h=\tau_{C,\mathcal{M}}$

3.)  $Rel^{\mathcal{M}}$ is factors through the pullback of the morphism $Rel^{\mathcal{N}}\rightarrow N^{[n]}$ along $h^{[n]}$. \demo

\end{definition}

If $h:\mathcal{M}\rightarrow \mathcal{N}$ is an $\mathcal{L}$-structure homomorphism, one easily checks that $P_{Term_{\mathcal{L}},\mathcal{M}}$ is the pullback of \textlbrackdbl\textrbrackdbl$_{Term_{\mathcal{L}},\mathcal{N}}:P_{Term_{\mathcal{L}},\mathcal{N}}\rightarrow N$ along $h$.   If $Form_{\mathcal{L}}$ denotes any of $Atom_{\mathcal{L}},Horn_{\mathcal{L}},Reg_{\mathcal{L}},$ or $Geom_{\mathcal{L}}$, we have a morphism $h_{Form_{\mathcal{L}}}:P_{Form_{\mathcal{L}},\mathcal{M}}\rightarrow P_{Form_{\mathcal{L}},\mathcal{N}}$ sending $(\varphi,\overrightarrow{m})$ to $(\varphi,h^{[n]}(\overrightarrow{m}))$.  Note that $h$ \textit{preserves the interpretation of} $Form_{\mathcal{L}}$, that is, $Q_{Form_{\mathcal{L}},\mathcal{M}}$ factors through the pullback of the inclusion of $Q_{Form_{\mathcal{L}},\mathcal{N}}$ along $h_{Form_{\mathcal{L}}}$, and we say $h$ \textit{reflects the interpretation of} $Form_{\mathcal{L}}$ if $Q_{Form_\mathcal{L},\mathcal{M}}$ is isomorphic to the pullback.

\begin{definition}
 
Let $\mathbb{T}$ be an internal geometric theory and let $\mathcal{M}$ and $\mathcal{N}$ be models of $\mathbb{T}$.  A $\mathbb{T}$-\textit{model homomorphism} is an $\mathcal{L}$-structure homomorphism $h:\mathcal{M}\rightarrow\mathcal{N}$. \demo

\end{definition}

Notice that, if $\mathbb{T}$ is an internal geometric theory and $h:\mathcal{M}\rightarrow\mathcal{N}$ is an $\mathcal{L}$-structure homomorphism that both preserves and reflects interpretations of all formulas, then $\mathcal{M}$ is a model of $\mathbb{T}$ if and only if $\mathcal{N}$ is.

If $\mathcal{L}$ is an internal language in a topos $\mathcal{E}$, $\mathbb{T}$ is an internal geometric theory over $\mathcal{L}$, $\mathcal{M}$ is a model of $\mathbb{T}$, and $f:\mathcal{F}\rightarrow\mathcal{E}$ is a geometric morphism, it is tedious, but not difficult to show that $f^{*}(\mathcal{M})$ is a model of $\mathbb{T}_{f^{*}(\mathcal{L})}$, and we therefore omit the proof.  The idea is that $P_{Term_{\mathcal{L}}}$, $P_{Form_{\mathcal{L}}}$ and $Q_{Form_{\mathcal{L}}}$ are all constructed by operations that are preserved by $f^{*}$, namely pullbacks, finite intersections of subobjects, images, and arbitrary unions of subobjects.  
 
\section{Classifying Topoi for Internal Geometric Theories}
\label{Classifying Topoi for Internal Geometric Theories}

Like any geometric theory, internal geometric theories have classifying toposes.  Here we present the classifying topos of an internal geometric theory as a certain sheaf subtopos of the topos of presheaves on the category of finite presentations of $\mathcal{L}$-structures, the classifying topos for $\mathcal{L}$-structures.  We shall also denote tuples of terms in the same context $(\overrightarrow{x}:t_{0},...,\overrightarrow{x}:t_{k})$ by $\overrightarrow{x}:\overrightarrow{t}$, and if $\overrightarrow{y}:\overrightarrow{u}$ has the same length and sort, $\overrightarrow{x}:\overrightarrow{t}=\overrightarrow{y}:\overrightarrow{u}$ denotes the Horn formula $\overrightarrow{x}\overrightarrow{y}:\bigwedge_{i=0}^{k}(t_{i}=:u_{i})$.  We will omit the context when referring to terms of the form $\overrightarrow{x}:\overrightarrow{x}$.  Our presentation follows \cite{BlaSc83}, where classifying topoi for Horn theories are constructed similarly.\bigskip

\begin{definition}
 
Let $\overrightarrow{x}:\overrightarrow{t}$ be a tuple of terms, let $\overrightarrow{y}$ be a context of the same length and sort, and let $\overrightarrow{y}:\phi$ be a Horn formula in the context $\overrightarrow{y}$; then $\overrightarrow{x}:\phi[\overrightarrow{t}/\overrightarrow{y}]$ denotes the Horn formula in the context $\overrightarrow{x}$ obtained by replacing every instance of each $y_{i}$ in $\phi$ by $t_{i}$.  If the original context of $\phi$ is clear, we sill sometimes just write $\overrightarrow{x}:\phi(\overrightarrow{t})$  In particular, if $\overrightarrow{t}=\overrightarrow{x}$, $\overrightarrow{x}:\phi[\overrightarrow{x}/\overrightarrow{y}]$ denotes the Horn formula obtained by replacing every instance of each $y_{i}$ by $x_{i}$.\demo

\end{definition}

\begin{definition}

Let $\mathcal{L}$ be an internal language.  The object of \textit{finite presentations} of $\mathcal{L}$-structures is $FP_{\mathcal{L}}=Horn_{\mathcal{L}}$ modulo the equivalence relation $\overrightarrow{x}:\phi\sim \overrightarrow{y}:\phi'$ whenever the contexts $\overrightarrow{x}$ and $\overrightarrow{y}$ have the same length and list of sorts and $\overrightarrow{x}:\phi=\overrightarrow{x}:\phi'[\overrightarrow{x}/\overrightarrow{y}]$. We will denote a finite presentation by the formal class term $\{\overrightarrow{x}:\phi\}$ and we think of the variables $\overrightarrow{x}$ as bound in the term.\demo

\end{definition}

$FP_{\mathcal{L}}$ can be given the structure of an internal category as follows:  We may assume the tuples $\overrightarrow{x}$ and $\overrightarrow{y}$ are disjoint. A morphism from $\{\overrightarrow{x}:\phi\}$ to $\{\overrightarrow{y}:\psi\}$ is given by an equivalence class of Horn formulas (under the equivalence defined above) $[\overrightarrow{y}=\overrightarrow{x}:\overrightarrow{t}]$ where the tuple $\overrightarrow{x}:\overrightarrow{t}$ of terms has the same sort as $\overrightarrow{y}$, and $\overrightarrow{x}:\psi(\overrightarrow{t})$, the Horn formula-in-context obtained by replacing each occurrence of $y_{i}$ in $\overrightarrow{y}:\psi$ by $t_{i}$,  is a subformula of $\overrightarrow{x}:\phi$.  Given two morphisms $[\overrightarrow{y}=\overrightarrow{x}:\overrightarrow{t}]:\{\overrightarrow{x}:\phi\}\rightarrow\{\overrightarrow{y}:\psi\}$ and $[\overrightarrow{z}=\overrightarrow{y}:\overrightarrow{u}]:\{\overrightarrow{y}:\psi\}\rightarrow\{\overrightarrow{z}:\chi\}$, their composition is $[\overrightarrow{z}=\overrightarrow{x}:\overrightarrow{u}(\overrightarrow{t})]$, where $\overrightarrow{x}:\overrightarrow{u}(\overrightarrow{t})$ is the tuple of terms obtained by replacing every instance of $y_{i}$ in $\overrightarrow{y}:\overrightarrow{u}$ by $t_{i}$.  One easily checks that $\overrightarrow{x}:\chi(\overrightarrow{t})$ is a subformula of $\overrightarrow{x}:\phi$.\bigskip

Each finite presentation $\{\overrightarrow{x}:\phi\}$ freely generates a finitely presented $\mathcal{L}$-structure $\langle\overrightarrow{x}:\phi\rangle$ as follows:  The underlying object is $\overrightarrow{x}\sqcup C$ with $\sigma=\tau_{Con}|_{\overrightarrow{x}}\sqcup\tau_{C}$ as the structure map, modulo the equivalence relation generated by the relation: $\overrightarrow{x}:t\sim \overrightarrow{x}:t'$ iff $\overrightarrow{x}:(t=t')$ is a subformula of $\overrightarrow{x}:\phi$.  (That is, the intersection of all equivalence relations that contain the given relation considered as subobjects of $(\overrightarrow{x}\sqcup C)^{2}$.)  Constant symbols are interpreted in the obvious way, and relation symbols are interpreted by saying that $R$ holds of (an appropiate sort) tuple $\overrightarrow{t}$ of variables and constant symbols if and only if $\overrightarrow{x}:\overrightarrow{t}$ is a term-in-context and $\overrightarrow{x}:R(\overrightarrow{t})$ is a subformula of $\overrightarrow{x}:\phi$ or there are terms-in-context $\overrightarrow{x}:\overrightarrow{u}$ such that $\overrightarrow{x}:R(\overrightarrow{u})$ and $\overrightarrow{x}:(\overrightarrow{t}=\overrightarrow{u})$ are subformulas of $\overrightarrow{x}:\phi$.  A morphism $\overrightarrow{y}=\overrightarrow{x}:\overrightarrow{t}:\{\overrightarrow{x}:\phi\}\rightarrow\{\overrightarrow{y}:\psi\}$ gives rise to a homomorphism $\langle\overrightarrow{y}:\psi\rangle\rightarrow\langle\overrightarrow{x}:\phi\rangle$ which sends the equivalence class of the element $u$ to the equivalence class of $u(\overrightarrow{t})$ which is $u$ if $u$ is a constant symbol and $t_{i}$ if $u$ is $y_{i}$. This makes $\langle\rangle$ into a full and faithful contravariant functor to the category of finitely presented $\mathcal{L}$-structures.\bigskip

The category $FP_{\mathcal{L}}$ has finite limits.  It has as its terminal object $\{:\top\}$, and if $[\overrightarrow{z}=\overrightarrow{x}:\overrightarrow{t}]:\{\overrightarrow{x}:\phi\}\rightarrow\{\overrightarrow{z}:\chi\}$ and $[\overrightarrow{z}=\overrightarrow{y}:\overrightarrow{u}]:\{\overrightarrow{y}:\psi\}\rightarrow\{\overrightarrow{z}:\chi\}$ are morphisms (we may assume $\overrightarrow{x}$ and $\overrightarrow{y}$ are disjoint), then their pullback is $\{\overrightarrow{x}\overrightarrow{y}:\phi\wedge\psi\wedge(\overrightarrow{t}=\overrightarrow{u})\}$ with the obvious projections.  In addition, $FP_{\mathcal{L}}$ comes equipped with a morphism $\tau:FP_{\mathcal{L}}\rightarrow S$ which sends an object $\{\overrightarrow{x}:\phi\}$ to the sort of $\overrightarrow{x}$.\bigskip

\begin{theorem}
The presheaf topos $\mathcal{E}^{FP_{\mathcal{L}}^{op}}$ classifies $\mathcal{L}$-structures.

\end{theorem}

\begin{proof}
By Diaconescu's theorem, geometric morphisms $f:\mathcal{F}\rightarrow\mathcal{E}^{FP_{\mathcal{L}}^{op}}$ correspond to internal flat functors $F:f^{*}(FP_{\mathcal{L}})\rightarrow\mathcal{F}$.  Working in the internal logic of $\mathcal{F}$, we show that such internal flat functors correspond to $\mathcal{L}$-structures in $\mathcal{F}$.  We will suppress mention of $f^{*}$, since the argument takes place entirely in $\mathcal{F}$.\bigskip

$(\Rightarrow)$  Let $F$ be an internal flat functor on $FP_{\mathcal{L}}.$  We construct an $\mathcal{L}$-structure $\mathcal{M}$ as follows:  The underlying object $M$ is $\bigcup_{s\in S}F(\{(x,s):\top\})$.  That is, $M$ is the union of the subobjects $F(\{x:\top\})$ of the object part of $F$, and this is well-defined since $\{x:\top\}\sim\{y:\top\}$ whenever $x$ and $y$ are contexts with the same sort.  Since $F$ preserves finite limits, the fiber of $M^{[n]}$ over $\overrightarrow{s}$ is (isomorphic to) $F(\{\overrightarrow{x}:\top\})$ where $\overrightarrow{x}$ has sort $\overrightarrow{s}$, and $F(:\top)$ is (isomorphic to) the terminal object of $\mathcal{F}$.  The interpretation $\tau_{C}$ of $C$ is the morphism which sends a constant symbol $c$ to the element $F([x=x:c])$ of $\{x:\top\}$.  The interpretation $Rel^{\mathcal{M}}$ of $Rel$ in $\mathcal{M}$ is the subobject of the pullback of $\sigma$ along $\tau_{Con}$ consisting of the subobjects $\{\overrightarrow{x}:R(\overrightarrow{x})\}\subseteq\{\overrightarrow{x}:\top\}$.\bigskip

$(\Leftarrow)$  Conversely, let $\mathcal{M}$ be an $\mathcal{L}$-structure in $\mathcal{F}$.  We define a flat functor $F$ on $FP_{\mathcal{L}}$ by induction on formulas as follows:  $F(\{\overrightarrow{x}:\top\})=\{\overrightarrow{m}\in M^{[n]}|\sigma^{[n]}(\overrightarrow{m})=\tau_{Con}(\overrightarrow{x})\}$ and $F(\{\overrightarrow{x}:\phi\})=$\textlbrackdbl$\overrightarrow{x}:\phi$\textrbrackdbl$_{Horn_{\mathcal{L}}}\subseteq F(\{\overrightarrow{x}:\top\})$.  If $[\overrightarrow{y}=\overrightarrow{x}:\overrightarrow{t}]:\{\overrightarrow{x}:\phi\}\rightarrow\{\overrightarrow{y}:\psi\}$ is a morphism and $\overrightarrow{m}\in F(\{\overrightarrow{x}:\phi\})$, then $F([\overrightarrow{y}=\overrightarrow{x}:\overrightarrow{t}])=\overrightarrow{t}(\overrightarrow{m})$, where $\overrightarrow{t}(\overrightarrow{m})$ is $m_{i}$ if $t=x_{i}$ and $\tau_{C}(c)$ if $t_{i}=c$.  It is easy to check that $F$ is a functor and that it preserves finite limits, hence is flat since $FP_{\mathcal{L}}$ has all finite limits.  It is also easy to verify that the two constructions outlined above are inverse to each other up to natural isomorphism.   

\end{proof}

\begin{definition}

If $\mathbb{C}$ is an internal category in a topos $\mathcal{E}$ the object of \textit{sieves} on $\mathbb{C}$ is the subobject $Sieve_{\mathbb{C}}=\{R\in P(C_{1})|\forall m,m'\in C_{1}(m\in R\wedge m'\in R\rightarrow cod(m)=cod(m'))$ and $\forall m,m'\in C_{1}(m\in R\wedge cod(m')=dom(m)\rightarrow m\circ m'\in R)\}$ interpreted in the internal logic of $\mathcal{E}$.  If $R$ is any subobject of $C_{1}$ satisfying $\forall m,m'\in C_{1}(m\in R\wedge m'\in R\rightarrow cod(m)=cod(m'))$, it generates a sieve $\tilde{R}$ on its codomain $c$ given by $\tilde{R} = \{m\in C_{1}|(cod(m)=c)\wedge\exists m'\in R,\exists m''\in C_{1}(m'\circ m'' = m)\}$.  A \textit{Grothendieck topology} \cite{Johns02} on $\mathbb{C}$ is an assignment $J$ to each $c\in C_{0}$ a family of sieves on $c$ (that is, $J$ is a morphism from $C_{0}$ to $P(Sieves_{\mathbb{C}})$ such that $R\in J(c)\rightarrow R$ is a sieve on $c$) such that:\bigskip

1.)  The maximal sieve $\{m\in C_{1}|cod(m)=c\}$ is in $J(c)$\bigskip

2.)  If $R\in J(c)$ and $m\in C_{1}$ has $cod(m)=c$ and $dom(m)=d$ then the pullback $\{m'\in C_{1}|cod(m')=d\wedge m\circ m'\in R\}$ of $R$ along $m$ is in $J(d)$\bigskip

3.)  If $R\in J(c)$ and $R'$ is a sieve on $c$ such that it is internally true that $\forall m\in R$ the pullback of $R'$ along $m$ is in $J(dom(m))$, then $R'\in J(c)$\demo

\end{definition}

Let $\mathbb{T}$ be an $\mathcal{L}$-theory.  We now construct the classifying topos for $\mathbb{T}$-models as a sheaf subtopos of $\mathcal{E}^{FP_{\mathcal{L}}^{op}}$ for a Grothendieck topology on $FP_{\mathcal{L}}$ defined as follows:  Suppose $\phi\rightarrow\bigvee_{\gamma\in\Gamma}\gamma$ is an axiom of $\mathbb{T}$.  For each $\overrightarrow{x}:\gamma=\overrightarrow{x}:\exists\overrightarrow{y}_{\gamma}\psi_{\gamma}$ we have a morphism $[\overrightarrow{x}=\overrightarrow{x}]_{\gamma}:\{\overrightarrow{x},\overrightarrow{y}_{\gamma}:\psi_{\gamma}\wedge\phi\}\rightarrow\{\overrightarrow{x}:\phi\}$.  Let $R_{(\phi,\Gamma)}$ denote the sieve generated by all such morphisms for $\gamma\in\Gamma$.\bigskip

If $R$ is a sieve on $\{\overrightarrow{z}:\chi\}$ in $FP_{\mathcal{L}}$, we say $R$ \textit{directly covers} $\{\overrightarrow{z}:\chi\}$ if either:\bigskip

1.)  $R$ is the maximal sieve on $\{\overrightarrow{z}:\chi\}$, or\bigskip

2.)  There is an axiom $\phi\rightarrow\bigvee_{\gamma\in\Gamma}\gamma$ with $\phi$ and $\Gamma$ formulas in the context $\overrightarrow{x}$ and $\overrightarrow{x}:\gamma=\overrightarrow{x}:\exists \overrightarrow{y}_{\gamma}(\psi_{\gamma}\wedge\phi)$, and there is a morphism $[\overrightarrow{x}=\overrightarrow{z}:\overrightarrow{t}]:\{\overrightarrow{z}:\chi\}\rightarrow\{\overrightarrow{x}:\phi\}$ such for every commutative diagram:

\[
\xymatrix{\{\overrightarrow{y}_{\gamma},\overrightarrow{z}:\psi_{\gamma}(\overrightarrow{t})\wedge\chi\}  \ar[dd]^-{[\overrightarrow{x}=\overrightarrow{z}:\overrightarrow{t}],[\overrightarrow{y}_{\gamma}=\overrightarrow{y}_{\gamma}]}  \ar[rr]^-{[\overrightarrow{z}=\overrightarrow{z}]_{\gamma}}  & &  \{\overrightarrow{z}:\chi\}  \ar[dd]^-{[\overrightarrow{x}=\overrightarrow{z}:\overrightarrow{t}]}\\
\\           
\{\overrightarrow{x},\overrightarrow{y}_{\gamma}:\psi_{\gamma}\wedge\phi\}  \ar[rr]^-{[\overrightarrow{x}=\overrightarrow{x}]_{\gamma}}  & &  \{\overrightarrow{x}:\phi\}}
\]\bigskip
the morphism $[\overrightarrow{z}=\overrightarrow{z}]$ is in $R$.  The notion of direct covering gives the closure of the family of sieves generated by the axioms of $\mathbb{T}$ under the first two requirements for a Grothendieck topology.  To obtain the closure under the third, we define a family $\Upsilon$ of sieves in $FP_{\mathcal{L}}$ to be \textit{covering closed} if, for every $R\in{\Upsilon}$ and every sieve $R'$ on the same object as $R$ such that the pullback of $R'$ along each morphism in $R$ directly covers its domain, we have $R'\in{\Upsilon}$.  Given any family $\Upsilon$, we define its covering closure $\bar{\Upsilon}$ to be the smallest family of sieves such that $\Upsilon\subseteq\bar{\Upsilon}$ and $\bar{\Upsilon}$ is covering closed.  It is easy to see that such a family exists:  The intersection of any collection of covering closed families is covering closed.  For each object $\{\overrightarrow{z}:\chi\}$ of $FP_{\mathcal{L}}$, let $\Upsilon_{\{\overrightarrow{z}:\chi\}}$ be the family of all sieves on $\{\overrightarrow{z}:\chi\}$ that directly cover $\{\overrightarrow{z}:\chi\}$.  Then $J(\{\overrightarrow{z}:\chi\})=\bar{\Upsilon}_{\{\overrightarrow{z}:\chi\}}$. \bigskip

\begin{theorem}
$Sh_{J}(\mathcal{E}^{FP_{\mathcal{L}}^{op}})$ is the classifying topos for $\mathbb{T}$. 

\end{theorem}

\begin{proof}
By Diaconescu's theorem, geometric morphisms $f:\mathcal{F}\rightarrow Sh_{J}(\mathcal{E}^{FP_{\mathcal{L}}^{op}})$ correspond to internal continuous flat functors $F:f^{*}(FP_{\mathcal{L}})\rightarrow\mathcal{F}$.  Working in the internal logic of $\mathcal{F}$ we show that such internal continuous flat functors correspond to models of $\mathbb{T}$ in $\mathcal{F}$.  Once again we suppress mention of $f^{*}$.\bigskip

Let $F$ be an internal flat functor on $FP_{L}$ and let $\mathcal{M}$ be the $\mathcal{L}$-structure constructed in theorem II.29.  Let $\phi\rightarrow\bigvee_{\gamma\in\Gamma}\gamma$ be an axiom of $\mathbb{T}$ with $\phi$ and $\Gamma$ formulas in the context $\overrightarrow{x}$ and $\overrightarrow{x}:\gamma=\overrightarrow{x}:\exists\overrightarrow{y}_{\gamma}(\psi_{\gamma}\wedge\phi)$. By construction \textlbrackdbl$\overrightarrow{x}:\phi$\textrbrackdbl$_{Horn_{\mathcal{L}}}=F(\{\overrightarrow{x}:\phi\})$.  It is clear from the definition of interpretation in $\mathcal{M}$ that \textlbrackdbl$\overrightarrow{x}:\bigvee_{\gamma\in\Gamma}\gamma$\textrbrackdbl$_{Geom_{\mathcal{L}}}=\{\overrightarrow{m}\in M^{[n]}|$ for some $\overrightarrow{x}:\gamma=\overrightarrow{x}:\exists \overrightarrow{y}_{\gamma}\psi_{\gamma}\in\Gamma$ and some $\overrightarrow{n}\in M^{[n]}$ of the same sort as $\overrightarrow{y}_{\gamma}$ we have $\overrightarrow{m},\overrightarrow{n}\in F(\{\overrightarrow{x},\overrightarrow{y}_{\gamma}:\psi_{\gamma}\wedge\phi\})\}$.\bigskip

$(\Rightarrow)$  We assume $F$ is continuous and show that $\mathcal{M}$ is a model of $\mathbb{T}$.  Let $\phi\rightarrow\bigvee_{\gamma\in\Gamma}\gamma$ be an axiom of $\mathbb{T}$.  For each $\overrightarrow{x}:\gamma=\overrightarrow{x}:\exists\overrightarrow{y}_{\gamma}\psi_{\gamma}\in\Gamma$ we have a morphism $[\overrightarrow{x}=\overrightarrow{x}]_{\gamma}:\{\overrightarrow{x},\overrightarrow{y}_{\gamma}:\psi_{\gamma}\wedge\phi\}\rightarrow\{\overrightarrow{x}:\phi\}$ and $F([\overrightarrow{x}=\overrightarrow{x}]_{\gamma})(\overrightarrow{a})=\overrightarrow{a}$, hence \textlbrackdbl$\overrightarrow{x}:\bigvee_{\gamma\in\Gamma}\gamma$\textrbrackdbl$_{Geom{\mathcal{L}}}\subseteq$\textlbrackdbl$\overrightarrow{x}:\phi$\textrbrackdbl$_{Horn_{\mathcal{L}}}$.  But $F$ is continuous and the $[\overrightarrow{x}=\overrightarrow{x}]_{\gamma}$ generate a cover, so the $F([\overrightarrow{x}=\overrightarrow{x}]_{\gamma})$ form a jointly epimorphic family; hence \textlbrackdbl$\bigvee_{\gamma\in\Gamma}\gamma$\textrbrackdbl$_{Geom{\mathcal{L}}}$ must be all of \textlbrackdbl$\overrightarrow{x}:\phi$\textrbrackdbl$_{Geom{\mathcal{L}}}$, and $\mathcal{M}\vDash\phi\rightarrow\bigvee_{\gamma\in\Gamma}\gamma$.  The result follows since the choice of axiom was arbitrary.\bigskip

Let $\mathcal{M}$ be an $\mathcal{L}$-structure and let $F$ be the internal flat functor constructed in theorem II.29.  Let $\phi\rightarrow\bigvee_{\gamma\in\Gamma}\gamma$ be an axiom of $\mathbb{T}$ with $\phi$ and $\Gamma$ formulas in the context $\overrightarrow{x}$ and $\overrightarrow{x}:\gamma=\overrightarrow{x}:\exists\overrightarrow{y}_{\gamma}(\psi_{\gamma}\wedge\phi)$.  By construction $F(\{\overrightarrow{x}:\phi\})=$ \textlbrackdbl$\overrightarrow{x}:\phi$\textrbrackdbl$_{Horn_{\mathcal{L}}}$.  We show that $\bigcup\{im(F([\overrightarrow{x}=\overrightarrow{x}]_{\gamma})|[\overrightarrow{x}=\overrightarrow{x}]_{\gamma}:\{\overrightarrow{x},\overrightarrow{y}:\psi_{\gamma}\wedge\phi\}\rightarrow\{\overrightarrow{x}:\phi\}\}=$\textlbrackdbl$\overrightarrow{x}:\bigvee_{\gamma\in\Gamma}\gamma)$\textrbrackdbl$_{Geom{\mathcal{L}}}$.  Indeed, suppose $\overrightarrow{m}$ belongs to the left hand side.  Then for some $\gamma\in\Gamma$ and some tuple $\overrightarrow{n}$ of the same type as $\overrightarrow{y}_{\gamma}$ we have $\overrightarrow{m},\overrightarrow{n}\in F(\{\overrightarrow{x},\overrightarrow{y}:\psi_{\gamma}\wedge\phi\}$) and $F([\overrightarrow{x}=\overrightarrow{x}]_{\gamma})(\overrightarrow{m},\overrightarrow{n})=\overrightarrow{n}$.  But this is precisely the condition for membership in the right hand side.\bigskip

$(\Leftarrow)$  We assume $\mathcal{M}$ is a model of $\mathbb{T}$ and show that $F$ is continuous.  Note that it suffices to check that suffices to check that $F$ turns covers induced by axioms into epimorphic families, since $F$ preserves pullbacks.  Let $R_{(\phi,\Gamma)}$ be a cover induced by an axiom of $\mathbb{T}$ as above.  Since $\mathcal{M}$ is a model of $\mathbb{T}$, we have that \textlbrackdbl$\overrightarrow{x}:\phi$\textrbrackdbl$_{Horn_{\mathcal{L}}}\subseteq$\textlbrackdbl$\overrightarrow{x}:\bigvee_{\gamma\in\Gamma}\gamma$\textrbrackdbl$_{Geom_{\mathcal{L}}}$.  But, by above, the left hand side is $F(\{\overrightarrow{x}:\phi\})$ and the right hand side is the union of the images of the morphisms in the cover, so that $F(R_{(\phi,\gamma)})$ is jointly epimorphic.  The result follows since the choice of axiom was arbitrary.  

\end{proof}

Since $FP_{\mathcal{L}}$ has a terminal object $\{:\top\}$, the classifying topos of a geometric theory $\mathbb{T}$ over $\mathcal{L}$ is degenerate precisely when the empty sieve on $\{:\top\}$ belongs to $J(\{:\top\})$, or equivalently, every covering closed family of sieves on $\{:\top\}$ contains the empty sieve.  For more on classifying topoi, see \cite[D3]{Johns02}.

%% file: chap2.tex
In this chapter we describe a limited infinitary deduction calculus for internal geometric theories and show that the classifying topos of an internal geometric theory is degenerate if and only if a contradiction is provable in the deduction calculus.  Central to our work here is the notion of an inductive construction, which we describe in the first section.  Our definition follows \cite{Blass87}.

\section{Inductive Constructions}

\begin{definition}
 
An \textit{inductive construction} in a topos $\mathcal{E}$ is a triple $(X,S,P,C)$, where $X$ is an object of $\mathcal{E}$ (the ambient set), $S$ is an object of $\mathcal{E}$ (the set of construction steps), $P$ is a subset of $X\times S$ ($(x,s)\in P$ means that $x$ is a prerequisite for applying construction step $s$), and $C:S\rightarrow X$ is a morphism (the application of the construction steps).

A subset $Y\subseteq X$ is said to be \textit{closed} for the inductive construction if $\forall s\in S(\forall x\in X((x,s)\in P\rightarrow x\in Y)\rightarrow C(s)\in Y)$ is true in the internal logic of $\mathcal{E}$.  It is easy to see that the intersection of any family of closed subobjects is closed, so, given any subobject $Y\subseteq X$ there is a unique smallest subobject $\bar{Y}$ of $X$ containing $Y$ which is closed for the inductive construction.  $\bar{Y}$ is called the \textit{closure of} $Y$ in $(X,S,P,C)$, and the closure of $\emptyset$ is called the \textit{closure of} $(X,S,P,C)$.  An inductive construction is called total if its closure is the whole of $X$.\demo
 
\end{definition}
  
\begin{theorem}
 
If $(X,S,P,C)$ is an inductive construction in a topos $\mathcal{E}$ and $f^{*}:\mathcal{F}\rightarrow\mathcal{E}$ is a geometric morphism, then if $(X,S,P,C)$ is total, so is $(f^{*}(X),f^{*}(S),f^{*}(P),f^{*}(C))$.

\end{theorem}

\begin{proof}

See \cite{Blass87}.

\end{proof}

\section{The Grothendieck Deduction Calculus for Internal Propositional Theories}
\label{The Grothendieck Deduction Calculus for Internal Propostitional Theories}

\begin{definition}
 
The Grothendieck rules of inference for internal propositional logic are:

1.)  $\vdash p\rightarrow p$

2.)  $\{p\rightarrow \bigvee_{q\in\alpha}q\}\vdash p'\rightarrow \bigvee_{q\in\alpha}(q\cup p')$ whenever $p\subseteq p'$

3.)  $\{p\rightarrow \bigvee_{q\in\alpha}q\}\cup\{q\rightarrow \bigvee_{r\in \alpha_{q}}r|q\in\alpha\}\vdash p\rightarrow \bigvee_{q\in\alpha, r\in\alpha_{q}}r$ and

$\{p\rightarrow\bot,p\rightarrow\bigvee_{q\in\alpha}q\}\vdash p\rightarrow\bigvee_{q\in\alpha}q$\demo  

\end{definition}

To formalize the notion of a proof in the Grothendieck deduction calculus of a sequent from an internal propositional theory $\mathbb{T}$ in a topos $\mathcal{E}$, we define an inductive construction as follows:

1.)  $X=Seq_{A}$

2.)  $S_{1}=\{p\rightarrow p|p\in K(A)\}\subseteq Seq_{A}$

$S_{2}=\{(p\rightarrow\bigvee_{q\in\alpha}q,p')|(p\rightarrow\bigvee_{q\in\alpha}q)\in Seq_{A}, p\subseteq p'\}\subseteq Seq_{A}\times K(A)$

$S_{3}=\{(p\rightarrow \bigvee_{q\in\alpha}q,\{q\rightarrow\bigvee_{r\in\alpha_{q}}r|q\in\alpha\})|\alpha,\alpha_{q}\in P(K(A))\}\sqcup Seq_{A}\subseteq Seq_{A}\times P(Seq_{A})\sqcup Seq_{A}$

and $S=S_{1}\sqcup S_{2}\sqcup S_{3}$

3.)  $P_{1}=\emptyset \subseteq X\times S_{1}$

$P_{2}=\{(\varphi,(p\rightarrow\bigvee_{q\in\alpha}q,p'))|\varphi=(p\rightarrow\bigvee_{q\in\alpha}q)\}\subseteq X\times S_{2}$

$P_{3}=\{(\varphi,(p\rightarrow \bigvee_{q\in\alpha}q,\{q\rightarrow\bigvee_{r\in\alpha_{q}}r|q\in\alpha\}))|\varphi=(p\rightarrow \bigvee_{q\in\alpha}q)$ or $\varphi\in\{q\rightarrow\bigvee_{r\in\alpha_{q}}r|q\in\alpha\}\}\sqcup\{(\varphi,p\rightarrow\bigvee_{q\in\alpha}q)|\phi=(p\rightarrow\bot)\}\subseteq X\times S_{3}$

$P=P_{1}\sqcup P_{2}\sqcup P_{3}\subseteq X\times S$

4.)  $C_{1}: S_{1}\rightarrow X$ is the inclusion

$C_{2}: S_{2}\rightarrow X$ is given by $C(p\rightarrow\bigvee_{q\in\alpha}q,p')=p'\rightarrow\bigvee_{q\in\alpha}(q\cup p')$

$C_{3}: S_{3}\rightarrow X$ is given by $C((p\rightarrow \bigvee_{q\in\alpha}q,\{q\rightarrow\bigvee_{r\in\alpha_{q}}r|q\in\alpha\}))=p\rightarrow\bigvee_{q\in\alpha,r\in\alpha_{q}}r$ and the identity on $Seq_{A}$

$C=C_{1}\sqcup C_{2}\sqcup C_{3}:S\rightarrow X$\demo

We say that a sequent is deducible from $\mathbb{T}$ in the Grothendieck deduction calculus if it belongs to the closure of $\mathbb{T}$ in the inductive construction above.  Notice that $\top\rightarrow\bot$ is deducible from $\mathbb{T}$ if and only if the closure of $\mathbb{T}$ is all of $Seq_{A}$.  We are now ready to prove:

\begin{theorem}

The classifying topos of $\mathbb{T}$ is degenerate if and only if $\top\rightarrow\bot$ is deducible from $\mathbb{T}$. 

\end{theorem}
   
\begin{proof}  $(\Rightarrow)$  Suppose the classifying topos of $\mathbb{T}$ is degenerate and let $Z=\{r\in K(A)|r\rightarrow\bot$ is in the closure of $\mathbb{T}\}$.  We show that $Z$ contains everything it directly covers.  Indeed, suppose $Z$ directly covers $r$.  Then either the maximal sieve on $r$ is contained in $Z$ and therefore $r\in Z$ or there is an axiom $p\rightarrow \bigvee_{q\in\alpha}q$ in $\mathbb{T}$ such that $p\subseteq r$ and $\forall q\in\alpha(q\cup r\in Z)$.  By the second rule of inference, $r\rightarrow \bigvee_{q\in\alpha}(q\cup r)$ is in the closure of $\mathbb{T}$, and each $q\cup r$ is in $Z$; hence for each $q$ in $\alpha$, $q\cup r\rightarrow\bot$ is in the closure of $\mathbb{T}$.  Thus, by the third rule of inference, $r\rightarrow \bot$ is in the closure of $\mathbb{T}$, so $r$ is in $Z$.  Since the classifying topos of $\mathbb{T}$ is degenerate, $\emptyset\in Z$, so $\top\rightarrow\bot$ is in the closure of $\mathbb{T}$ as desired.

$(\Leftarrow)$  Suppose $\top\rightarrow\bot$ is in the closure of $\mathbb{T}$.  Let $Y=\{p\rightarrow \bigvee_{q\in\alpha}q|\forall r\in K(A)(p\subseteq r\wedge\forall q\in\alpha(q\cup r\in\bar{\emptyset})\rightarrow r\in\bar{\emptyset})\}$.  We show that $Y$ is closed for the inductive construction:

1.)  Trivial

2.)  Suppose $(p\rightarrow\bigvee_{q\in\alpha}q)\in Y$ and $p\subseteq p'$.  If $p'\subseteq r$ and $\forall q\in\alpha(q\cup p'\cup r=q\cup r\in \bar{\emptyset})$ then since $p\subseteq r$ we must have $r\in \bar{\emptyset}$.  Thus $p'\rightarrow \bigvee_{q\in\alpha}q\cup p'\in Y$ since $r$ was arbitrary.

3.)  Suppose $(p\rightarrow\bigvee_{q\in\alpha}q)\in Y$ and $\forall q\in\alpha((q\rightarrow \bigvee_{r\in\alpha_{q}}r)\in Y)$.  If $p\subseteq s$ and $\forall q\in\alpha\forall r\in\alpha_{q}(r\cup s\in\bar{\emptyset})$, then also $\forall q\in\alpha,\forall r\in\alpha_{q}(r\cup q\cup s\in \bar{\emptyset})$ since $\bar{\emptyset}$ is upward closed.  But then $\forall q\in\alpha(q\cup s\in \bar{\emptyset})$ and therefore $s\in \bar{\emptyset}$.  Since $s$ was arbitrary, we must have $(p\rightarrow\bigvee_{q\in\alpha,r\in\alpha_{q}}r)\in Y$.

Since $\bar{\emptyset}$ contains everything it directly covers, $\mathbb{T}\subseteq Y$.  Then since $Y$ is closed, $(\top\rightarrow\bot)\in Y$.  But then $\emptyset\in\bar{\emptyset}$ as desired.
 
\end{proof}

\section{The Grothendieck Deduction Calculus for Internal Geometric Theories}
\label{The Groethendieck Deduction Calculus for Internal Geometric Theories}

With a little more work, we can show the same is true for internal geometric theories.

\begin{definition}
 
The Grothendieck deduction calculus for $GeomSeq_{\mathcal{L}}$ has the following rules:

1.)  $\vdash \phi\rightarrow\phi$

2.)  $\{\phi\rightarrow\bigvee_{\gamma\in\Gamma}\gamma\} \vdash \phi'\rightarrow\bigvee_{\gamma\in\Gamma}(\gamma(\overrightarrow{t})\wedge\phi')$ whenever $\overrightarrow{t}$ is a tuple of terms which are either constants or free variables of $\phi'$ and $\phi(\overrightarrow{t})$ is a subformula of $\phi'$

3.)  $\{\phi\rightarrow\bigvee_{\gamma\in\Gamma}\gamma\}\cup\{\psi_{\gamma}\rightarrow\bigvee_{\lambda\in\Lambda_{\gamma}}\lambda|\gamma\in\Gamma, \gamma=\exists\overrightarrow{y}_{\gamma}\psi_{\gamma}\} \vdash \phi\rightarrow\bigvee_{\gamma\in\Gamma \lambda\in\Lambda_{\gamma}}\lambda$ and

$\{(\phi\rightarrow\bot,\phi\rightarrow\bigvee_{\gamma\in\Gamma}\gamma\}\vdash\phi\rightarrow\bigvee_{\gamma\in\Gamma}\gamma$\demo 

\end{definition}
  
Given an internal geometric theory $\mathbb{T}$, we define an inductive construction as follows:

1.)  $X=GeomSeq_{L}$

2.)  $S1=\{\phi\rightarrow\phi|\phi\in Horn_{\mathcal{L}}\}$

$S2=\{(\phi\rightarrow\bigvee_{\gamma\in\Gamma}\gamma,\phi',\overrightarrow{t})|(\phi\rightarrow\bigvee_{\gamma\in\Gamma}\gamma)\in(GeomSeq_{\mathcal{L}}$ and $\overrightarrow{t}\in Term_{\mathcal{L}}^{[n]}$ is such that $\phi(\overrightarrow{t})$ is a subformula of $\phi'\}$

$S3=\{(\phi\rightarrow\bigvee_{\gamma\in\Gamma}\gamma,\{\psi_{\gamma}\rightarrow\bigvee_{\lambda\in\Lambda_{\gamma}}\lambda|\gamma\in\Gamma, \gamma=\exists\overrightarrow{y}_{\gamma}\psi_{\gamma}\})|\phi\in Horn_{\mathcal{L}},(\bigvee_{\gamma\in\Gamma}\gamma),(\bigvee_{\lambda\in\Lambda_{\gamma}}\lambda)\in Geom_{\mathcal{L}}\}\sqcup GeomSeq_{\mathcal{L}}$

$S=S1\sqcup S2\sqcup S3$

3.)  $P1=\emptyset\subseteq X\times S1$

$P2=\{(\varphi,(\phi\rightarrow\bigvee_{\gamma\in\Gamma}\gamma,\phi',\overrightarrow{t}))|\varphi=(\phi\rightarrow\bigvee_{\gamma\in\Gamma}\gamma)\}\subseteq X\times S2$

$P3=\{(\varphi,(\phi\rightarrow\bigvee_{\gamma\in\Gamma}\gamma,\{\psi_{\gamma}\rightarrow\bigvee_{\lambda\in\Lambda_{\gamma}}\lambda|\gamma\in\Gamma, \gamma=\exists\overrightarrow{y}_{\gamma}\psi_{\gamma}\}))|\varphi=(\phi\rightarrow\bigvee_{\gamma\in\Gamma}\gamma)$ or $\varphi\in\{\psi_{\gamma}\rightarrow\bigvee_{\lambda\in\Lambda_{\gamma}}\lambda|\gamma\in\Gamma, \gamma=\exists\overrightarrow{y}_{\gamma}\psi_{\gamma}\}\}\sqcup\{(\varphi,\phi\rightarrow\bigvee_{\gamma\in\Gamma}\gamma)|\varphi=(\phi\rightarrow\bot\}$

$P=P1\sqcup P2\sqcup P3$

4.)  $C1: S1\rightarrow X$ is the inclusion

$C2: S2\rightarrow X$ is given by $C((\phi\rightarrow\bigvee_{\gamma\in\Gamma}\gamma,\phi',\overrightarrow{t}))=\phi'\rightarrow\bigvee_{\gamma\in\Gamma}(\gamma(\overrightarrow{t})\wedge\phi')$

$C3: S3\rightarrow X$ is given by $C3((\phi\rightarrow\bigvee_{\gamma\in\Gamma}\gamma,\{\psi_{\gamma}\rightarrow\bigvee_{\lambda\in\Lambda_{\gamma}}\lambda|\gamma\in\Gamma, \gamma=\exists\overrightarrow{y}_{\gamma}\psi_{\gamma}\}))=\phi\rightarrow\bigvee_{\gamma\in\Gamma \lambda\in\Lambda_{\gamma}}\lambda$ and the identity on $GeomSeq_{\mathcal{L}}$

$C=C1\coprod C2\coprod C3$

See \cite{Johns02} for a description of the full deduction calculus for geometric logic.  We say a sequent is deducible from $\mathbb{T}$ in the Grothendieck deduction calculus if it belongs to the closure of $\mathbb{T}$ under this inductive construction.  Notice that $\top\rightarrow\bot$ is deducible from $\mathbb{T}$ iff the closure of $\mathbb{T}$ is all of $GeomSeq_{L}$.

\begin{theorem}
 
The classifying topos for $\mathbb{T}$ is degenerate iff $\top\rightarrow\bot$ is deducible from $\mathbb{T}$.

\end{theorem}

\begin{proof}
 
$(\Rightarrow)$  Suppose the classifying topos of $\mathbb{T}$ is degenerate and let $Z=\{R|R$ is a sieve on $\{:\top\}$ and there is a subfamily $\{\{\overrightarrow{z}_{i}:\chi_{i}\}\rightarrow\{:\top\}|i\in I\}\subseteq R$ such that $\top\rightarrow\bigvee_{i\in I}\exists\overrightarrow{z}_{i}\chi_{i}$ is deducible from $\mathbb{T}\}$.  We show $Z$ is covering closed.  Indeed, suppose that there is an $R\in Z$ and a sieve $R'$ on $\{:\top\}$ such that the pullback of $R'$ along each morphism $r\in R$ directly covers the domain of $r$.  Then for each $i\in I$ there is an axiom $\phi_{i}\rightarrow\bigvee_{\gamma\in\Gamma_{i}}\gamma$, where we may assume $\gamma=\exists\overrightarrow{y}_{\gamma}(\psi_{\gamma}\wedge\phi_{i})$, and a morphism $[\overrightarrow{x}_{i}=\overrightarrow{t}_{i}]:\{\overrightarrow{z}_{i}:\chi_{i}\}\rightarrow\{\overrightarrow{x}_{i}:\phi_{i}\}$ such that for every commutative diagram:

\[
\xymatrix{\{\overrightarrow{y}_{\gamma},\overrightarrow{z}_{i}:\psi_{\gamma}(\overrightarrow{t}_{i})\wedge\chi_{i}\}  \ar[dd]^-{[\overrightarrow{x}_{i}=\overrightarrow{t}_{i}],[\overrightarrow{y}_{\gamma}=\overrightarrow{y}_{\gamma}]}  \ar[rr]^-{[\overrightarrow{z}_{i}=\overrightarrow{z}_{i}]} & &  \{\overrightarrow{z}_{i}:\chi_{i}\}  \ar[dd]^-{[\overrightarrow{x}_{i}=\overrightarrow{t}_{i}]}\\
\\           
\{\overrightarrow{x}_{i},\overrightarrow{y}_{\gamma}:\psi_{\gamma}\wedge\phi_{i}\}  \ar[rr]^-{[\overrightarrow{x}_{i}=\overrightarrow{x}_{i}]} & &  \{\overrightarrow{x}_{i}:\phi_{i}\}}
\]  
the morphism $[\overrightarrow{z}_{i}=\overrightarrow{z}_{i}]$ is in the pullback of $R'$ along $\{\overrightarrow{z}_{i}:\chi_{i}\}\rightarrow\{:\top\}$.  Thus for every $i\in I$ the morphism $\{\overrightarrow{y}_{\gamma},\overrightarrow{z}_{i}:\psi_{\gamma}(\overrightarrow{t}_{i})\wedge\chi_{i}\}\rightarrow\{:\top\}$ is in $R'$.  Now since $\phi_{i}(\overrightarrow{t}_{i})$ is a subformula of $\chi_{i}$, we have that $\chi_{i}\rightarrow\bigvee_{\gamma\in\Gamma_{i}}(\gamma(\overrightarrow{t}_{i})\wedge\chi_{i})$ is in the closure of $\mathbb{T}$ by the second rule of inference.  But then $\top\rightarrow\bigvee_{i\in I,\gamma\in\Gamma_{i}}\gamma(\overrightarrow{t}_{i})\wedge\chi_{i}$ is in the closure of $\mathbb{T}$ by the third rule of inference; hence $R'$ is in $Z$.  Since $R'$ was arbitrary, $Z$ is covering closed.  But the classifying topos for $\mathbb{T}$ is degenerate, so the empty sieve is in $Z$ and therefore $\top\rightarrow\bot$ is deducible from $\mathbb{T}$.\bigskip

$(\Leftarrow)$  Suppose that $\top\rightarrow\bot$ is deducible from $\mathbb{T}$ and let $Y=\{\phi\rightarrow\bigvee_{\gamma\in\Gamma}\gamma\in GeomSeq_{\mathcal{L}}$ where we may assume $\gamma=\exists\overrightarrow{y}_{\gamma}(\psi_{\gamma}\wedge\phi)|$ For every object $\{\overrightarrow{z}:\chi\}$ of $FP_{\mathcal{L}}$, every sieve $R$ on $\{\overrightarrow{z}:\chi\}$, and every morphism $[\overrightarrow{x}=\overrightarrow{t}]:\{\overrightarrow{z}:\chi\}\rightarrow\{\overrightarrow{x}:\phi\}$, if the pullback of $R$ along $[\overrightarrow{z}=\overrightarrow{z}]_{\gamma}$ in each of the diagrams:

\[
\xymatrix{\{\overrightarrow{y}_{\gamma},\overrightarrow{z}:\psi_{\gamma}(\overrightarrow{t})\wedge\chi\}  \ar[dd]^-{[\overrightarrow{x}=\overrightarrow{t}],[\overrightarrow{y}_{\gamma}=\overrightarrow{y}_{\gamma}]}  \ar[rr]^-{[\overrightarrow{z}=\overrightarrow{z}]_{\gamma}} & &  \{\overrightarrow{z}:\chi\}  \ar[dd]^-{[\overrightarrow{x}=\overrightarrow{t}]}\\
\\           
\{\overrightarrow{x},\overrightarrow{y}_{\gamma}:\psi_{\gamma}\wedge\phi\}  \ar[rr]^-{[\overrightarrow{x}=\overrightarrow{x}]_{\gamma}} & &  \{\overrightarrow{x}:\phi\}}
\]
is in $J(\{\overrightarrow{y}_{\gamma},\overrightarrow{z}:\psi_{\gamma}(\overrightarrow{t})\wedge\chi\})$, then $R\in J(\{\overrightarrow{z}:\chi\})\}$.  We show that $Y$ is closed for the inductive construction:

1.)  Trivial

2.)  Suppose $\phi\rightarrow\bigvee_{\gamma\in\Gamma}\gamma\in Y$ and $\overrightarrow{u}$ is a tuple of terms with variables among the free variables of $\phi'$  such that $\phi(\overrightarrow{u})$ is a subformula of $\phi'$.  If the pullback of a sieve $R$ on $\{\overrightarrow{z}:\chi\}$ along $[\overrightarrow{z}=\overrightarrow{z}]_{\gamma}$ in each of the diagrams:

\[
\xymatrix{\{\overrightarrow{y}_{\gamma},\overrightarrow{z}:\psi_{\gamma}(\overrightarrow{t})\wedge\chi\}  \ar[dd]^-{[\overrightarrow{w}=\overrightarrow{t}],[\overrightarrow{y}_{\gamma}=\overrightarrow{y}_{\gamma}]}  \ar[rr]^-{[\overrightarrow{z}=\overrightarrow{z}]_{\gamma}} & &  \{\overrightarrow{z}:\chi\}  \ar[dd]^-{[\overrightarrow{x}=\overrightarrow{t}]}\\
\\           
\{\overrightarrow{w},\overrightarrow{y}_{\gamma}:\psi_{\gamma}\wedge\phi'\}  \ar[rr]^-{[\overrightarrow{w}=\overrightarrow{w}]_{\gamma}} & &  \{\overrightarrow{w}:\phi'\}}
\]

Is in $J(\{\overrightarrow{y}_{\gamma},\overrightarrow{z}:\psi_{\gamma}(\overrightarrow{t})\wedge\chi\})$ we form the composite diagram:

\[
\xymatrix{\{\overrightarrow{y}_{\gamma},\overrightarrow{z}:\psi_{\gamma}(\overrightarrow{t}(\overrightarrow{u}))\wedge\chi\}  \ar[dd]^-{[\overrightarrow{w}=\overrightarrow{t}],[\overrightarrow{y}_{\gamma}=\overrightarrow{y}_{\gamma}]}  \ar[rr]^-{[\overrightarrow{z}=\overrightarrow{z}]_{\gamma}} & &  \{\overrightarrow{z}:\chi\}  \ar[dd]^-{[\overrightarrow{x}=\overrightarrow{t}]}\\
\\           
\{\overrightarrow{w},\overrightarrow{y}_{\gamma}:\psi_{\gamma}(\overrightarrow{u})\wedge\phi'\}  \ar[dd]^{[\overrightarrow{x}=\overrightarrow{u}],[\overrightarrow{y}_{\gamma}=\overrightarrow{y}_{\gamma}]}  \ar[rr]^-{[\overrightarrow{w}=\overrightarrow{w}]_{\gamma}} & &  \{\overrightarrow{w}:\phi'\}  \ar[dd]^{[\overrightarrow{x}=\overrightarrow{u}]}\\ 
\\
\{\overrightarrow{x},\overrightarrow{y}_{\gamma}:\psi_{\gamma}\wedge\phi\}  \ar[rr]^{[\overrightarrow{x}=\overrightarrow{x}]_{\gamma}}  & &  \{\overrightarrow{x}:\phi \} }
\]
which witnesses that $R$ is in $J(\{\overrightarrow{z}:\chi\})$.  Thus $\phi'\rightarrow\bigvee_{\gamma\in\Gamma}\gamma(\overrightarrow{u})\wedge\phi'\in Y$

3.)  Suppose $\phi\rightarrow\bigvee_{\gamma\in\Gamma}\gamma\in Y$ and for each $\gamma\in\Gamma$, we have $\psi_{\gamma}\rightarrow\bigvee_{\lambda\in\Lambda_{\gamma}}\lambda\in Y$.  If the pullback of a sieve $R$ on $\{\overrightarrow{z}:\chi\}$ along $[\overrightarrow{z}=\overrightarrow{z}]_{\gamma,\lambda}$ in each of the diagrams:

\[
\xymatrix{\{\overrightarrow{w}_{\lambda},\overrightarrow{y}_{\gamma},\overrightarrow{z}:\lambda\wedge\psi_{\gamma}(\overrightarrow{t})\wedge\chi\}  \ar[dd]^-{[\overrightarrow{w}_{\lambda}=\overrightarrow{w}_{\lambda}],[\overrightarrow{x}=\overrightarrow{t}],[\overrightarrow{y}_{\gamma}=\overrightarrow{y}_{\gamma}]}  \ar[rr]^-{[\overrightarrow{z}=\overrightarrow{z}]_{\gamma}}  &  &  \{\overrightarrow{z}:\chi\}  \ar[dd]^-{[\overrightarrow{x}=\overrightarrow{t}]}\\
\\           
\{\overrightarrow{w}_{\lambda},\overrightarrow{x},\overrightarrow{y}_{\gamma}:\lambda\wedge\psi_{\gamma}\wedge\phi\}  \ar[rr]^-{[\overrightarrow{x}=\overrightarrow{x}]_{\gamma,\lambda}}  &  &  \{\overrightarrow{x}:\phi\}}
\]
is in $J(\{\overrightarrow{w}_{\lambda},\overrightarrow{y}_{\gamma},\overrightarrow{z}:\lambda\wedge\psi_{\gamma}(\overrightarrow{t})\wedge\chi\})$ then the pullback of the pullback of $R$ to $\{\overrightarrow{y}_{\gamma},\overrightarrow{z}:\psi_{\gamma}(\overrightarrow{t})\wedge\chi\}$ along each of the morphisms $[\overrightarrow{y}_{\gamma}=\overrightarrow{y}_{\gamma}]_{\lambda},[\overrightarrow{z}=\overrightarrow{z}]_{\lambda}$ in the diagrams:

\[
\xymatrix{\{\overrightarrow{w}_{\lambda},\overrightarrow{y}_{\gamma},\overrightarrow{z}:\lambda\wedge\psi_{\gamma}\wedge\chi\}  \ar[dd]^-{[\overrightarrow{x}=\overrightarrow{t}],[\overrightarrow{y}_{\gamma}=\overrightarrow{y}_{\gamma}],[\overrightarrow{z}=\overrightarrow{z}]}  \ar[rrr]^-{[\overrightarrow{y}_{\gamma}=\overrightarrow{y}_{\gamma}]_{\lambda},[\overrightarrow{z}=\overrightarrow{z}]_{\lambda}}  &  &  &  \{\overrightarrow{y}_{\gamma},\overrightarrow{z}:\psi_{\gamma}(\overrightarrow{t})\wedge\chi\}  \ar[dd]^-{[\overrightarrow{x}=\overrightarrow{t}],[\overrightarrow{y}_{\gamma}=\overrightarrow{y}_{\gamma}]}\\
\\
\{\overrightarrow{w}_{\lambda}, \overrightarrow{x}, \overrightarrow{y}_{\gamma}:\lambda\wedge\psi_{\gamma}\wedge\phi\}  \ar[rrr]^-{[\overrightarrow{x}=\overrightarrow{x}]_{\lambda},[\overrightarrow{y}_{\gamma}=\overrightarrow{y}_{\gamma}]_{\lambda}}  &  &  &  \{\overrightarrow{x},\overrightarrow{y}_{\gamma}:\psi_{\gamma}\wedge\phi\}} 
\]
is in $J(\{\overrightarrow{w}_{\lambda},\overrightarrow{y}_{\gamma},\overrightarrow{z}:\lambda\wedge\psi_{\gamma}(\overrightarrow{t})\wedge\chi\})$ as well; hence the pullback of $R$ to $\{\overrightarrow{y}_{\gamma},\overrightarrow{z}:\psi_{\gamma}(\overrightarrow{t})\wedge\chi\}$ is in $J(\{\overrightarrow{y}_{\gamma},\overrightarrow{z}:\psi_{\gamma}(\overrightarrow{t})\wedge\chi\})$.  But then the pullback of $R$ along each of the morphisms $[\overrightarrow{z}=\overrightarrow{z}]_{\gamma}$ in the diagrams:

\[
\xymatrix{\{\overrightarrow{y}_{\gamma},\overrightarrow{z}:\psi_{\gamma}(\overrightarrow{t})\wedge\chi\}  \ar[dd]^-{[\overrightarrow{x}=\overrightarrow{t}],[\overrightarrow{y}_{\gamma}=\overrightarrow{y}_{\gamma}]}  \ar[rr]^-{[\overrightarrow{z}=\overrightarrow{z}]_{\gamma}} & &  \{\overrightarrow{z}:\chi\}  \ar[dd]^-{[\overrightarrow{x}=\overrightarrow{t}]}\\
\\           
\{\overrightarrow{x},\overrightarrow{y}_{\gamma}:\psi_{\gamma}\wedge\phi\}  \ar[rr]^-{[\overrightarrow{x}=\overrightarrow{x}]_{\gamma}} & &  \{\overrightarrow{x}:\phi\}}
\]
is in $J(\{\overrightarrow{y}_{\gamma},\overrightarrow{z}:\psi_{\gamma}(\overrightarrow{t})\wedge\chi\}$; hence $R$ is in $J(\{\overrightarrow{z}:\chi\})$.  It follows that $(\phi\rightarrow\bigvee_{\gamma\in\Gamma,\lambda\in\Lambda_{\gamma}}\lambda)\in Y$.

Clearly $\mathbb{T}\subseteq Y$; hence, since $\top\rightarrow\bot$ is deducible from $\mathbb{T}$ and $Y$ is closed, $\top\rightarrow\bot\in Y$.  Then, by the condition for membership in $Y$, for every sieve $R$ on $\{\overrightarrow{z}:\chi\}$ and every morphism $\{\overrightarrow{z}:\chi\}\rightarrow\{:\top\}$, $R$ is in $J(\{\overrightarrow{z}:\chi\}$.  In particular, the empty sieve covers every object, so the classifying topos of $\mathbb{T}$ is degenerate.

\end{proof}

\section{Downward Stratified Inductive Constructions}
\label{Downward Stratified Inductive Constructions}

From the previous sections and from the definition it seems that inductive constructions are somewhat proof theoretic in nature.  One is naturally drawn to ask whether every inductive construction arises from the degeneracy of the classifying topos of a geometric theory.  The obvious approach to this question doesn't work, see the construction and counterexample in \cite{Blass87}, but we have the following partial result:

\begin{definition}

An inductive construction $(X,S,P,C)$ is called \textit{downward stratified} if there is a partial order $\leq$ on the set $X$ such that $\forall x,y\in X,\forall s\in S(P(x,s)\wedge C(s)=y\rightarrow x\leq y)$ and for every $y\in X$, every construction step $s\in S$ such that $C(s)=y$, and every $y'\leq y$ there is a construction step $s'$ such that $C(s')=y'$ and the downclosure of $\{x\in X|P(x,s)\}$ contains $\{x'\in X|P(x',s')\}$.\demo 

\end{definition}
 
Let $(X,S,P,C)$ be a downward stratified inductive construction in a topos $\mathcal{E}$ and consider the following internal site:  The underlying internal category is $X$ considered as an internal poset, and for each $x\in X$ and construction step $s$ that produces $x$, $x$ is covered by the sieve generated by $\{y\in X|P(y,s)\}$.  Since $(X,S,P,C)$ is downward stratified, the pullback of a basic cover along any morphism contains a basic cover.  Analogously to the situation for propositional theories above, we say a subset $R$ of $X$ directly covers $x\in X$ if it contains a basic cover of $x$.  Given a subset $R$ of $X$ we define $\bar{R}$ to be the smallest subset of $X$ that contains $R$ and everything it directly covers.  Then the Grothendieck topology $J$ generated by the basic covers is given by $J(x)=\{R\subseteq X|R$ is a sieve on $x$ and $x\in\bar{R}\}$.  By Diaconescu's theorem, $Sh(X,J)$ is the classifying topos for the internal geometric theory of internal continuous flat functors on $(X,J)$ and we have the following theorem:

\begin{theorem}
 
$(X,S,P,C)$ is total if and only if $Sh(X,J)$ is degenerate. 

\end{theorem}

\begin{proof}

Let $Y$ denote the closure of $(X,S,P,C)$.  $\bar{\emptyset}$ is the intersection of all of the subsets of $X$ that contain everything they directly cover.  If $x\in \bar{\emptyset}$ and $x'\leq x$, then, being the smallest set that contains everything it directly covers, $\bar{\emptyset}$ must directly cover $x$.  (Let $Z=\{y\in\bar{\emptyset}|\bar{\emptyset}$ directly covers $y\}$.  If $Z$ directly covers $y$ so does $\bar{0}$ since $Z\subseteq\bar{\emptyset}$; hence $y\in Z$.  Since $Z$ contains everything it directly covers we must also have $\bar{\emptyset}\subseteq Z$.)  Since $(X,S,P,C)$ is downward stratified, $\bar{\emptyset}$ must also directly cover $x'$; hence $x'\in\bar{\emptyset}$.  Thus $\bar{\emptyset}$ is downward closed.  If $s\in S$ is a construction step and $\forall y\in X(P(y,s)\rightarrow y\in\bar{\emptyset})$ then $\bar{\emptyset}$ contains the downclosure of $\{y\in X|P(y,s)\}$, so $\bar{\emptyset}$ directly covers $C(s)$.  Then $C(s)\in\bar{\emptyset}$; hence $\bar{\emptyset}$ is closed, since $s\in S$ was arbitrary.  Therefore $Y\subseteq\bar{\emptyset}$.  Conversely, if $Y$ directly covers $x\in X$, it contains a set of prerequisites for $x$; hence $x\in Y$ since $Y$ is closed.  Thus $Y$ contains everything it directly covers, so $\bar{\emptyset}\subseteq Y$.  If the construction is total, then $Y=X$ and therefore every element of $X$ is covered by the empty sieve, so the classifying topos is degenerate.  Conversely, if the classifying topos is degenerate, $\bar{\emptyset}=X$, hence every element of $X$ belongs to the closure, so the construction is total. 

\end{proof}

%% file: chap3.tex
In this chapter we show that if $p:A\rightarrow B$ is an epimorphism in a Boolean topos $\mathcal{E}$, there is a topos $\mathcal{F}$ and a surjective geometric morphism $f:\mathcal{F}\rightarrow\mathcal{E}$ such that $f^{*}(p)$ is split in $\mathcal{F}$.  For Grothendieck topoi, not necessarily Boolean, this follows from Barr's theorem \cite[Thm7.57]{Johns77}.  However, Barr's theorem makes an essential use of the axiom of choice in $Sets$ and is therefore not suitable for our purposes.

Let $p:A\rightarrow B$ be an epimorphism in a non-degenerate Boolean topos $\mathcal{E}$.  We define an internal propositional theory $\mathbb{S}$ whose models are sections of $p$ as follows:  The sentence symbols are the members of $B\times A$ and the axioms are:\bigskip

1.)  $\top\rightarrow\bigvee_{a\in A}(b,a)$\hfill for each $b\in B$\bigskip

2.)  $(b,a)\wedge(b,a')\rightarrow$ \textlbrackdbl $a=a'$ \textrbrackdbl \hfill for each $a,a' \in A$, $b\in B$\bigskip

3.)  $(b,a)\rightarrow$ \textlbrackdbl $p(a)=b$ \textrbrackdbl \hfill for each $a\in A$, $b\in B$\bigskip

where \textlbrackdbl$x=y$\textrbrackdbl is shorthand for $\bigvee\{\top|x=y\}$.

Since $p$ is an epimorphism, a model of this theory will be the graph of a morphism from $B$ to $A$ which is a section of $p$.  We show that when $B$ is inhabited then the classifying topos $\mathcal{E}[\mathbb{S}]$ of this theory is non-degenerate, so the inverse image of $p$ under the canonical geometric morphism from $\mathcal{E}[\mathbb{S}]$ to $\mathcal{E}$ has a section, namely the universal model of $\mathbb{S}$.  Moreover, we show that this geometric morphism is a surjection.\bigskip

The classifying topos for $\mathbb{S}$ is constructed as follows:  The underlying (internal) category of the site of definitions is the object $K(B\times A)$ of K-finite subobjects of $B\times A$ considered as an (internal) poset with the opposite of the inclusion order.  The first axiom scheme tells us that for each $b\in B$ the sieve generated by the family $\{\{(b,a)\}|a\in A\}$ of singletons covers the terminal object $\emptyset \in K(B\times A)$.  The second axiom scheme tells us that for each $a,a'\in A$, $b\in B$ $\{(b,a),(b,a')\}$ is covered by the sieve which is equal to the principal sieve to the extent that $a=a'$ and is empty to the extent that $a\neq a'$.  In particular, if $a\neq a'$, $\{(b,a),(b,a')\}$ is covered by the empty sieve.  The final axiom scheme tells us that singleton $\{(b,a)\}$ is covered by the sieve which is equal to the principal sieve to the extent that $p(a)=b$ and is empty to the extent that $p(a)\neq b$.  In particular, if $p(a)\neq b$, $\{(b,a)\}$ is covered by the empty sieve.  The Grothendieck topology on $K(B\times A)$ is then obtained by closing these covers under pullback and iteration.\bigskip

We give a description of the covers thus obtained.  We say a subset $R\subseteq K(B\times A)$ \textit{directly covers} $x\in K(B\times A)$ if one of the following holds:\bigskip

1.)  $\exists b\in B$ such that $x\cup \{(b,a)\}\in R$ for all $a\in A$.\bigskip

2.)  $\exists a,a'\in A$, $b\in B$ such that $(b,a), (b,a')\in x$ and $a=a'\rightarrow x\in R$  \bigskip 

3.)  $\exists a\in A$, $b\in B$ such that $(b,a)\in x$ and $p(a)=b\rightarrow x\in R$.\bigskip 

The notion of direct covering, when restricted to sieves, gives the pullback closure of the covers described above.  To get the full Grothendieck topology, let $\bar{R}$ be the smallest subset of $K(B\times A)$ containing $R$ and every element $x$ that it directly covers.  Then the Grothendieck topology on $K(B\times A)$ can be described as follows:  $J(x)=\{R\subseteq K(B\times A)|R$ is a sieve on $x$ and $x\in \bar{R}\}$. We will require the following:

\begin{definition}
 
Let $\mathbb{C}$ be an internal category in a topos $\mathcal{E}$ with internal Grothendieck topology $J$.  A subcategory $\mathbb{D}$ of $\mathbb{C}$ is called $J$-\textit{dense} \cite{Johns77} if:

1.)  Every object $c\in\mathbb{C}$ has a covering sieve $R\in J(c)$ generated by morphisms whose domains are in $\mathbb{D}$, and

2.)  For any morphism $g:c\rightarrow d$ in $\mathbb{C}$ with $d\in\mathbb{D}$ there is a covering sieve $R\in J(c)$ generated by morphisms $h:b\rightarrow c$ for which the composite $g\circ h$ is in $\mathbb{D}$. 

If $\mathbb{D}$ is $J$-dense, $J$ restricts to a Grothendieck topology $J_{\mathbb{D}}$ on $\mathbb{D}$ by setting $J_{\mathbb{D}}(d)=\{R\cap\mathbb{D}|R\in J(d)\}$.\demo

\end{definition}

\begin{lemma}
 
(Comparison Lemma)  Let $\mathbb{C}$ and $J$ be as above and let $\mathbb{D}$ be a $J$-dense subcategory.  Then there is a restriction functor $Sh_{J}(\mathcal{E}^{\mathbb{C}^{op}})\rightarrow Sh_{J_{\mathbb{D}}}(\mathcal{E}^{\mathbb{D}^{op}})$ which is one half of an equivalence of categories.

\end{lemma}

\begin{proof}
 
See \cite{ArGrV63} ``Lemme de Comparaison''.

\end{proof}

We're now ready to prove our main theorem of this chapter.

\begin{theorem}

If B is inhabited, the classifying topos of $\mathbb{S}$ is non-degenerate, and the geometric morphism $\mathcal{E}[\mathbb{S}]\rightarrow \mathcal{E}$ is surjective. 

\end{theorem}

\begin{proof}

We will use the Comparison Lemma to show that the site for $\mathcal{E}[\mathbb{S}]$ as described above is equivalent to a subsite which is inhabited and in which every covering sieve is principal.  In particular, every covering sieve is inhabited, which implies surjectivity.\bigskip

First we note that, since the first axiom scheme tells us that, for each $b\in B$, $\emptyset$ is covered by the sieve generated by the family of singletons $\{(b,a)|a\in A\}$ which is inhabited since $B$ is, and since every K-finite set is either empty or inhabited, every element of $K(B,A)$ is covered by inhabited sets.  Therefore, by the comparison lemma, we may reduce to the subsite consisting of the inhabited K-finite subsets with the induced topology.\bigskip

Next we note that the second axiom scheme tells us that if $x\in K(B,A)$ is inhabited, then for all $b\in B$ $x$ is covered by the sieve $S(x,b,a,a')=\{f|codom(f)=x$, $(b,a),(b,a')\in x$, and $a=a'\}$.  Now since $x$ is K-finite so is its square, and there are K-finitely members of $A$ such that $(b,a)\in x$, and the intersection of K-finitely many covering sieves is a covering sieve, we have that $S(x,b)=\bigcap_{(b,a),(b,a')\in x}S(x,b,a,a')=\{f|codom(f)=x$ and all the members of $x$ with first projection $b$ are equal$\}$ is a covering sieve.  Again since $x$ is K-finite, there are K-finitely many $b\in B$ in the first projection of $x$, hence the sieve $S(x)=\bigcap_{b\in \pi_{1}(x) }S(x,b)=\{f|codom(f)=x$ and all the members of $x$ with the same first projection are equal$\}$ is also a cover.  Note that we can shorten the description of $S(x)$ to $\{f|codom(f)=x$ and $x$ is a finite partial function from $B$ to $A\}$.  Now we can use the comparison lemma to reduce to the subsite of $K(B,A)$ consisting of those members which are finite partial functions:  If there is anything in $S(x)$ at all, then $x$ is a function and $S(x)$ is the principal sieve on $x$, hence every member of $S(x)$ factors through the identity on $x$.\bigskip

Finally, we note that the third axiom scheme tells us that if $x\in K(B,A)$ is a K-finite partial function, then for all $b\in B$ $x$ is covered by the sieve $S'(x,b,a)=\{f|codom(f)=x$ and $(b,a)\in x$ and $p(a)=b\}$.  As above, we obtain a covering sieve $S'(x)=\bigcap_{(b,a)\in x}S(x,b,a)=\{f|codom(f)=x$ and $p(a)=b$ for all $(b,a)\in x \}$.  We can simplify this last description to $\{f|codom(f)=x$ and $x$ is a K-finite partial section of $p\}$.  Now we can use the comparison lemma to reduce to the subsite of K-finite partial sections of $p$:  If there is anything in $S'(x)$ then $x$ is a K-finite partial section, and $S(x)$ is the principal sieve on $x$, hence every member of $S(x)$ factors through the identity on $x$.\bigskip

We have now shown that the site of definition for $\mathcal{E}[\mathbb{S}]$ is equivalent, via the comparison lemma, to the subsite consisting of those members of $K(B,A)$ which are K-finite partial sections of $p$.  Since $\mathcal{E}$ is Boolean, every K-finite subset of $B$ has a section of (the approprtiate restriction of) $p$ over it, hence our equivalent site is inhabited since $B$ is.  Now it is evident that the covers induced by the second and third axiom schemes yield principal sieves when restricted to this subsite, so it remains to show that the first axiom scheme gives covers which restrict to inhabited covers in our subsite.  But this is obvious:  If $x$ is a K-finite partial section of $p$, then the pullback of the cover of $\emptyset$ by singletons with first projection $b$ to $x$ restricted to our new site consists precisely of the set of K-finite partial sections of $p$ over $\pi_{1}(x)\cup \{b\}$ extending $x$, which is certainly inhabited.    
 
\end{proof}

Even if $\mathcal{E}$ is not Boolean, $f^{*}(p)$ splits in $\mathcal{E}[\mathbb{S}]$, but the geometric morphism $\mathcal{E}[\mathbb{S}]\rightarrow\mathcal{E}$ need not be surjective:  It is not generally internally true in a non-Boolean topos that epimorphisms with K-finite codomain split, so the covers described above need not be inhabited.

%% file: chap4.tex
\section{Dedekind Finite Objects in Topoi}
\label{Dedekind Finite Objects in Topoi}

\begin{definition}

An object $A$ in a topos $\mathcal{E}$ is said to be \textit{Dedekind finite} if it is internally true that every monomorphism $m:A\rightarrow A$ is an isomorphism. \demo

An object $A$ in a topos $\mathcal{E}$ is said to be \textit{Dedekind infinite} if it is internally true that there is a monomorphism $m:A\rightarrow A$ such that the complement of the image of $m$ is inhabited.\demo

\end{definition}

Dedekind finite objects in topoi have been studied in \cite{Stout87}.  If $A$ is not Dedekind finite then there is a monomorphism $m$ for which it is not internally true that $m$ is an isomorphism.  It need not be the case that such an $m$ misses a point, however, every topos $\mathcal{E}$ admits a surjective geometric morphism $f:\mathcal{B}\rightarrow\mathcal{E}$ from a Boolean topos \cite[4.5.23]{Johns02}, and since it is not internally true that in $\mathcal{E}$ that $m$ is an isomorphism, it is not internally true in $\mathcal{B}$ that $f^{*}(m)$ is, but then it is internally true in $\mathcal{B}$ that $m$ is not an isomorphism, and therefore the image of $f^{*}(m)$ is complemented in $f^{*}(A)$, that is, $f^{*}(m)$ misses a point and $f^{*}(A)$ is Dedekind infinite.  Thus a sufficient condition for $A$ to be Dedekind finite is that whenever $f:\mathcal{F}\rightarrow\mathcal{E}$ is a geometric morphism, if $f^{*}(m)$ misses a point, then $\mathcal{F}$ is degenerate.  This condition is not necessary, as we shall see in section 4.4 below; the pullback of a Dedekind finite object along a geometric morphism can be Dedekind infinite.  With this motivation in mind, we define what we shall call geometrically Dedekind finite objects in the next section. 

\section{Definition and Properties}
\label{Definition and Properties}

\begin{definition}
  
Let $A$ be an object in a topos $\mathcal{E}$.  Let $\mathcal{T}$($A$) be the geometric propositional theory whose sentence symbols are the members of $A\sqcup (A \times A)$ and whose axioms are: \bigskip

1.)  $\top\rightarrow\bigvee_{b\in A}(a,b)$  \hfill for every $a\in A$ \bigskip

2.)  $(a,b)\wedge(a,b')\rightarrow$\textlbrackdbl$b=b'$\textrbrackdbl \hfill for every $a,b,b'\in A$ \bigskip

3.)  $(a,b)\wedge(a',b)\rightarrow$\textlbrackdbl $a=a'$\textrbrackdbl \hfill for every $a,a',b\in A$ \bigskip

4.)  $b\wedge(a,b)\rightarrow\bot$ \hfill for every $a,b\in A$ \bigskip

5.)  $\top\rightarrow\bigvee_{b\in A}b$ \bigskip

\noindent where \textlbrackdbl$x=y$\textrbrackdbl \enskip is shorthand for for $\bigvee\{\top|x=y\}$.  
\end{definition}
\bigskip
A model of $\mathcal{T}$($A$) in $\mathcal{E}$ consists of a subobject $X$ of $A$ and a subobject $\alpha$ of $A \times A$.  Intuitively, axioms 1 and 2 assert that $\alpha$ is the graph of a function from $A$ to $A$, axiom 3 asserts that $\alpha$ is monic, and axioms 4 and 5 assert that $\alpha$ misses a point.  $X$ is the set of points missed by $\alpha$. 

\bigskip  

\begin{definition}
We say that an object $A$ in a topos $\mathcal{E}$ is \textit{geometrically Dedekind finite} (GDF for short) if, for every geometric morphism $f:\mathcal{F}\rightarrow\mathcal{E}$ such that $f^{*}(A)$ admits a model of $f^{*}(\mathcal{T}(A))$, $\mathcal{F}$ is degenerate.
\end{definition}
\bigskip
Like any geometric theory, $\mathcal{T}$($A$) has a classifying topos containing a universal model.  An object $A$ in a topos $\mathcal{E}$ is GDF if and only if the classifying topos for $\mathcal{T}(A)$ is degenerate. 
\bigskip
\begin{theorem}
Let $\mathcal{E}$ be a topos with a natural numbers object $\mathbb{N}$.  Then an object $A$ in $\mathcal{E}$ is GDF if and only if, for every geometric morphism $f:\mathcal{B}\rightarrow\mathcal{E}$ with $\mathcal{B}$ Boolean, if there exists a monomorphism $m:f^{*}(\mathbb{N})\rightarrow f^{*}(A)$ in $\mathcal{B}$, then $\mathcal{B}$ is degenerate.  
\end{theorem}
\begin{proof}
Let $A$ be GDF and let $f:\mathcal{B}\rightarrow\mathcal{E}$ be a geometric morphism with $\mathcal{B}$ Boolean.  Suppose there is a monomorphism $m:f^{*}(\mathbb{N})\rightarrow f^{*}(A)$ in $\mathcal{B}$.  Since $\mathcal{B}$ is Boolean, im($m$) is a complemented subobject of $f^{*}(A)$.  We define a monomorphism $\alpha : f^{*}(A)\rightarrow f^{*}(A)$ by setting $\alpha$ equal to the identity on complement of im($m$) and equal to $m\circ f^{*}(s)$ on im($m$), where $s:\mathbb{N}\rightarrow\mathbb{N}$ is the (monic) successor morphism on $\mathbb{N}$.  Then the graph of $\alpha$ together with the point $m\circ f^{*}(0)$ of $f^{*}(A)$ constitute a model of $f^{*}(\mathcal{T}(A))$, hence $\mathcal{B}$ is degenerate.

\bigskip

Conversely, suppose that $A$ satisfies the latter condition and let $f:\mathcal{F}\rightarrow\mathcal{E}$ be a geometric morphism such that $f^{*}(A)$ admits a model of $f^{*}(\mathcal{T}(A))$.  Then $f^{*}(A)$ contains a subobject which is a natural numbers object for $\mathcal{F}$ \cite[5.1.1]{Johns02}, hence there is a monomorphism $m:f^{*}(\mathbb{N})\rightarrow f^{*}(A)$ since $f^{*}(\mathbb{N})$ is an nno for $\mathcal{F}$.  See \cite[Lemma2.5.6]{Johns02}.  Let $p:\mathcal{B}\rightarrow\mathcal{F}$ be a surjection with $\mathcal{B}$ Boolean.  Then $p^{*}(m):p^{*}f^{*}(\mathbb{N})\rightarrow p^{*}f^{*}(A)$ is a monomorphism, hence $\mathcal{B}$ is degenerate.  Since $p^{*}$ reflects isomorphisms, $\mathcal{F}$ is degenerate.  The result follows since $f$ was arbitrary.     

\end{proof}

\bigskip

In fact, we can do better than this.  

\begin{theorem}
 
If $A$ is a GDF object in a Boolean topos, then $A$ is (internally) K-finite.

\end{theorem}

\begin{proof}

Suppose $A$ is a GDF object in a Boolean topos $\mathcal{B}$ that is not K-finite.  Then if $p$ is a K-finite subset of $A$, the complement of $p$ is inhabited.  Let $\mathcal{T}(\mathbb{N},A)$ be the internal propositional theory whose sentence symbols are the members of $\mathbb{N}\times A$ and whose axioms are:

1.)  $(n,a)\wedge(n,a')\rightarrow$\textlbrackdbl$a=a'$\textrbrackdbl \hfill for every $n\in\mathbb{N}$, $a,a'\in A$

2.)  $(n,a)\wedge(n',a)\rightarrow$\textlbrackdbl$n=n'$\textrbrackdbl \hfill for every $n,n'\in\mathbb{N}$, $a\in A$ 

3.)  $\top\rightarrow\bigvee_{a\in A}(n,a)$ \hfill for every $n\in\mathbb{N}$

Intuitively, models of $\mathcal{T}(\mathbb{N},A)$ are monomorphisms from $\mathbb{N}$ to $A$.  Similarly to the construction of splittings for epimorphisms, we apply the comparison lemma to find that the classifying topos for $\mathcal{T}(\mathbb{N},A)$ is equivalent to the topos of sheaves on the site whose underlying category consists of (internally) genuine K-finite parital monomorphisms from $\mathbb{N}$ to $A$ and whose basic covers of each K-finite partial $p$ are given by, for each $n\in\mathbb{N}$ not in the domain of $p$, the sieve generated by the extensions of $p$ to $n$.  Such extensions always exist, since the complement of the image of $p$ is inhabited by hypothesis.  It is clear that no new covers are obtained by pullback, and that iteration will not produce any empty covers since none of the basic covers are empty.  Therefore the classifying topos $\mathcal{E}(\mathcal{T}(\mathbb{N},A))$ surjects onto $\mathcal{E}$ and one easily checks that the universal model is (the graph of) a total monomorphism into the pullback of $A$.  Passing to double negation sheaves on the classifying topos, we arrive at a Boolean topos with a monomorphism from the natural numbers into the pullback of $A$, so double negation sheaves on the classifying topos, and therefore the classifying topos itself, must be degenerate.  (Since $\neg\neg0=0$ in any Heyting algebra, $\neg\neg$ is the trivial operator iff $0=1$.)  But then $\mathcal{E}$ is degenerate.  

\end{proof}

\begin{definition}

An object $A$ in a topos $\mathcal{E}$ is called \textit{geometrically finite} if whenever $f:\mathcal{B}\rightarrow\mathcal{E}$ is a geometric morphism with $\mathcal{B}$ Boolean, $f^{*}(A)$ is K-finite.  Geometric finiteness was first formulated in \cite{Freyd06}.\demo 

\end{definition}

\begin{corollary}

An object $A$ in a topos $\mathcal{E}$ is GDF if an only if it is geometrically finite. In particular, every K-finite, \~{K}-finite, and R-finite object is GDF.  See \cite{Freyd06} for the definitions of the latter notions. 

\end{corollary}

We now study some properties of GDF objects 

\begin{lemma}
Let $A$ and $B$ be objects in a Boolean topos $\mathcal{B}$ with natural numbers object $\mathbb{N}$.  If there is a monomorphism $m:\mathbb{N}\rightarrow A\times B$ then either there is a monomorphism from $\mathbb{N}$ to $A$ or there is a monomorphism from $\mathbb{N}$ to $B$.  
\end{lemma}
\begin{proof}

First we show that, for any morphism $f:\mathbb{N}\rightarrow X$ to an object $X$ in $\mathcal{B}$, either there is a monomorphism $g:\mathbb{N}\rightarrow\mathbb{N}$ such that $f\circ g$ is a monomorphism, or there is a natural number $i$ such that for every $n\in\mathbb{N}$ there exists an $m<i$ such that $f(n)=f(m)$.  We attempt to inductively define such a $g$ by setting $g(0)=0$ and $g(n+1)=$ the least $k>g(n)$ such that $f(k)$ is different from $f(m)$ for all $m\leqq g(n)$.  The inductive process might fail at some stage becuase there might be no $k$ to serve as $g(n+1)$ and we have the second alternative; otherwise the induction succeeds, and provides the $g$ in the first alternative.

\bigskip

Now to prove the lemma, suppose we have a monomorphism from $\mathbb{N}$ to $A\times B$ and let $f_{1}:\mathbb{N}\rightarrow A$ and $f_{2}:\mathbb{N}\rightarrow B$ be its composites with the respective projections.  Apply the above to both of these morphisms.  If either one of them yields the first alternative, we have a monomorphism as desired.  Otherwise both of them yield the second alternative, say, with $i_{1}$ for $f_{1}$ and $i_{2}$ for $f_{2}$.  Then for every $n\in\mathbb{N}$ there are $m_{1}<i_{1}$ and $m_{2}<i_{2}$ such that $(f_{1}(n),f_{2}(n))=(f_{1}(m_{1}),f_{2}(m_{2}))$.  Then our monomorphism from $\mathbb{N}$ into $A\times B$ maps $\mathbb{N}$ into a finite set of size $i_{1}i_{2}$, which is absurd. 
\end{proof}

\bigskip

\begin{corollary}
In any topos $\mathcal{E}$ with a natural numbers object $\mathbb{N}$:

i.)  Any K-finite object is GDF.\newline\indent
ii.)  The product of GDF objects is GDF.\newline\indent
iii.)  The coproduct of GDF objects is GDF.\newline\indent
iv.)  A subobject of a GDF object is GDF.\newline\indent
v.)  A quotient of a GDf object is GDF.
\end{corollary}
\begin{proof}
i.)  Let $A$ be a K-finite object in $\mathcal{E}$ and let $f:\mathcal{B}\rightarrow\mathcal{E}$ be a geometric morphism with $\mathcal{B}$ Boolean.  Suppose that $m:f^{*}(\mathbb{N})\rightarrow f^{*}(A)$ is a monomorphism.  It is well known that K-finiteness is preserved by inverse images, hence $f^{*}(A)$ is K-finite.  Then $f^{*}(\mathbb{N})$ is a subobject of a K-finite decidable object, hence K-finite, and decidable since $\mathcal{B}$ is Boolean.  But any monomorphism from a K-finite decidable object to itself is an isomorphism.  See \cite{Ortega92}.  Thus $f^{*}(s)$ is an isomorphism, so that $\mathcal{B}$ is degenerate.

\bigskip

ii.)  Let $A$ and $B$ be GDF and let $f:\mathcal{B}\rightarrow\mathcal{E}$ be a geometric morphism with $\mathcal{B}$ Boolean.  Suppose there is a monomorphism $m:f^{*}(\mathbb{N})\rightarrow f^{*}(A\times B)\cong f^{*}(A)\times f^{*}(B)$.  Then, by lemma V.8, either there is a monomorphism from $f^{*}(\mathbb{N})$ to $f^{*}(A)$ or a monomorphism from $f^{*}(\mathbb{N})$ to $f^{*}(B)$.  In either case $\mathcal{B}$ is degenerate, since $A$ and $B$ are GDF.

\bigskip

iii.)  Let $A$ and $B$ be GDF and let $f:\mathcal{B}\rightarrow\mathcal{E}$ be a geometric morphism with $\mathcal{B}$ Boolean.  Suppose there is a monomorphism $m:f^{*}(\mathbb{N})\rightarrow f^{*}(A\sqcup B)\cong f^{*}(A)\sqcup f^{*}(B)$.  Let $\mathbb{N}_{A}$ and $\mathbb{N}_{B}$ be the pullbacks along $m$ of the respective coproduct inclusions.  Then $f^{*}(\mathbb{N})\cong \mathbb{N}_{A}\sqcup\mathbb{N}_{B}$, hence one of $\mathbb{N}_{A}$ and $\mathbb{N}_{B}$ must be unbounded.  Thus there is either a monomorphism from $\mathbb{N}$ to $\mathbb{N}_{A}$ or a monomorphism from $\mathbb{N}$ to $\mathbb{N}_{B}$, therefore either a monomorphism from $\mathbb{N}$ to $f^{*}(A)$ or a monomorphism from $\mathbb{N}$ to $f^{*}(B)$.  It follows that $\mathcal{B}$ is degenerate since $A$ and $B$ are GDF.

\bigskip  

iv.)  Let $A'\rightarrowtail A$ be a subobject of a GDF object $A$ and let $f:\mathcal{B}\rightarrow\mathcal{E}$ be a geometric morphism with $\mathcal{B}$ Boolean.  Then $f^{*}(A')$ is a subobject of $f^{*}(A)$ and a monomorphism $m:f^{*}(\mathbb{N})\rightarrow f^{*}(A')$ yields, by composition with the inclusion, a monomorphism from $f^{*}(\mathbb{N})$ to $f^{*}(A)$, hence $\mathcal{B}$ is degenerate since $A$ is GDF. 

\bigskip

v.)  Let $p:A\rightarrow B$ be an epimorphism with $A$ GDF and let $f:\mathcal{F}\rightarrow\mathcal{E}$ be a surjective geometric morphism such that $f^{*}(p)$ has a splitting $s$ in $\mathcal{F}$, which exists by theorem III.1.  If $m:\mathbb{N}\rightarrow f^{*}(B)$ is a monomorphism, so is $s\circ m:\mathbb{N}\rightarrow f^{*}(A)$.  Pulling $s\circ m$ back to $Sh_{\neg\neg}(\mathcal{F})$ we find that $Sh_{\neg\neg}(\mathcal{F})$ is degenerate.  Since the inclusion of $Sh_{\neg\neg}(\mathcal{F})$ in $\mathcal{F}$ is dense, $\mathcal{F}$ is degenerate.  But $f$ is surjective, so $\mathcal{E}$ is degenerate.  

\end{proof}

\section{GDF is not a Geometric Theory}
\label{GDF is not a Geometric Theory}

In this section we will show that GDF is not Morita equivalent to a geometric theory.  We will use the following theorem from \cite{Caram11}:

\begin{theorem}
(Caramello)  Let $\Sigma$ be a signature and $\mathcal{S}$ a class of $\Sigma$-structures in Grothendieck toposes closed under isomorphisms of structures.  Then $\mathcal{S}$ is the class of all models in Grothendieck toposes of a geometric theory over $\Sigma$ if and only if it satisfies the following two conditions: \bigskip

i.)  For any geometric morphism $f:\mathcal{F}\rightarrow\mathcal{E}$, if $M\in \Sigma$-\textbf{str}$(\mathcal{E})$ is in $\mathcal{S}$ then $f^{*}(M)$ is in $\mathcal{S}$. \bigskip

ii.)  For any (set indexed) jointly surjective family $\{f_{i}:\mathcal{E}_{i}\rightarrow\mathcal{E}|i\in I\}$ of geometric morhpisms (that is, if $m$ is a morphism of $\mathcal{E}$ such that for all $i\in I$ $f_{i}^{*}(m)$ is an isomorphism then $m$ is) and any $\Sigma$-structure $M$ in $\mathcal{E}$, if $f_{i}^{*}(M)$ is in $\mathcal{S}$ for every $i\in I$ then $M$ is in $\mathcal{S}$.

\end{theorem}

\bigskip

If $X$ is a sober topological space, it is easy to show that the family $\{x:Sets\rightarrow Sh(X)|x\in X\}$ of points of $X$ form a jointly surjective family.  The inverse image functors $x^{*}$ for $x\in X$ are exactly the ``take the stalk at $x$'' functors, hence, by theorem V.10, a sheaf on $X$ is a model of a geometric theory $\mathbb{T}$ in $Sh(X)$ if and only if its stalks are models of $\mathbb{T}$ in $Sets$.\bigskip

By Remark 3.2 of \cite{Caram11}, the second condition of Theorem V.10 can be reformulated as follows:\bigskip

For any surjective geometric morphism $f:\mathcal{S}\rightarrow\mathcal{E}$ and any $M\in \Sigma-\textbf{str}(\mathcal{E})$, if $f^{*}(M)$ is in $\mathcal{S}$ then $M$ is in $\mathcal{S}$, and for any (set indexed) family $\{M_{i}|i\in I\}$ of $\Sigma$-structures in toposes $\mathcal{E}_{i}$ all of which are in $\mathcal{S}$, the structure in the coproduct $\coprod_{i\in I}\mathcal{E}_{i}$ whose $i^{th}$ coordinate is $M_{i}$ is also in $\mathcal{S}$.\bigskip

We assume that GDF is a geometric theory and give an example of a pair of sober spaces $X$ and $Y$, a continuous injection $f:X\rightarrow Y$ (whose induced geometric morphism $f:Sh(X)\rightarrow Sh(Y)$ is a surjection), and a sheaf $P$ on $Y$ such that $f^{*}(P)$ is GDF but $P$ is not.  This contradicts the reformulation of condition ii.) of Theorem 2.1.  The construction proceeds as follows:\bigskip

\begin{theorem}
Let $X$ be the topological space whose underlying set is the ordinal $\omega$ and whose topology is discrete, let $Y$ be the topological space whose underlying set is the ordinal $\omega + 1$ and whose topology is given by $\{[n,\omega]|n<\omega\}\cup\{\emptyset\}$, let $f:X\rightarrow Y$ be the inclusion of $\omega$ in $\omega + 1$ (so $f$ is continuous, since $X$ is discrete) and let $P$ be the presheaf on $Y$ defined as follows:\bigskip

$P([n,\omega])=n+1$\bigskip

$P(\emptyset)=1$\bigskip

$P([m,\omega]\subseteq [n,\omega])$ is the inclusion $n+1 \subseteq m+1$\bigskip

$P(\emptyset \subseteq [n,\omega])$ is the unique function $n+1\rightarrow 1$\bigskip

\noindent Then $X$ and $Y$ are sober, $f$ induces a surjective geometric morphism $f:Sh(X)\rightarrow Sh(Y)$, $P$ is a sheaf on $Y$, and $f^{*}(P)$ is GDF, but $P$ is not.

\end{theorem}
\begin{proof}
$X$ is clearly sober.  $Y$ is sober since every non-empty closed subset is irreducible, and those not of the form $Y$ are of the form $[0,n]$ for some $n\in \omega$, which is uniquely the closure of $\{n\}$.  $Y$ itself is the closure of $\{\omega\}$.\bigskip

$f^{-1}: \mathcal{O}(Y)\rightarrow\mathcal{O}(X)$ is clearly an injective frame homomorphism, hence $f$ is a surjective morphism of locales.  Therefore $f:Sh(X)\rightarrow Sh(Y)$ is a surjective geometric morphism, by the theory of geometric morphisms between categories of sheaves on a locale.  See \cite[C1.5]{Johns02} for proofs of these facts.\bigskip

To see that $P$ is a sheaf, we note that the only covering sieve on an open set $[n,\omega]\subseteq Y$ is the principal one.  Given a matching family $\{s_{m}|m\geq n\}\cup\{s_{\emptyset}\}$ for $P$ on the principal cover of $[n,\omega]$ we must have $P([m,\omega]\subseteq [n,\omega])(s_{n})=s_{m}$ for all $m\geq n$ and $s_{\emptyset}=*$, so the family is uniquely determined by $s_{n}$, hence $P$ is a sheaf since $n$ was arbitrary.\bigskip

Pulling $P$ back to $\mathcal{O}(X)$ (as a presheaf), which we denote $\tilde{f}^{*}(P)$, we have that $\tilde{f}^{*}(P)(U)=\underrightarrow{Colim}_{U\subseteq f^{-1}(V)}P(V)$, which is just $\inf(U)+1$ if $U$ is non-empty and $\omega$ if $U=\emptyset$.\bigskip

To sheafify $\tilde{f}^{*}(P)$, we note that the stalk at $x\in X$ of $\tilde{f}^{*}(P)$ is $\underrightarrow{Colim}_{x\in U}\tilde{f}^{*}(P)(U)$, which is just $x+1$.  Then the \'{e}tale space of $\tilde{f}^{*}(P)$ is $\coprod_{x\in X}(x+1)$ with the usual topology, which is irrelevant since $X$ is discrete and every section of the projection is continuous.  Thus the set of sections of the projection over an open set $U\subseteq X$ is $\prod_{x\in U}(x+1)$, that is, the sheafification of $\tilde{f}^{*}(P)$ at $U$ is given by $f^{*}(P)(U)=\prod_{x\in U}(x+1)$.  It is easy to verify that $f^{*}(P)(V\subseteq U)$ is just the product projection $\prod_{x\in U}(x+1)\rightarrow\prod_{x\in V}(x+1)$.\bigskip

Now $\{x\}$ is the limit (=intersection) of all the open sets $U$ with $x\in U$, hence the stalk of $f^{*}(P)$ at $x$ is $\underrightarrow{Colim}_{x\in U}f^{*}(P)(U)=f^{*}(P)(\underrightarrow{Lim}_{x\in U}U)=f^{*}(P)(\{x\})=x+1$.  Under the assumption that GDF is a geometric theory, since each stalk of $f^{*}(P)$ is a finite set, hence GDF, we must have that $f^{*}(P)$ is GDF.\bigskip

But $P$ is not GDF since the stalk of $P$ at $\omega$ in $Y$ is $\omega$, hence not GDF, contradicting the reformulation of condition ii.) in theorem 2.1.  Thus GDF is not a geometric theory.

\end{proof}

\section{Examples}
\label{Examples}

In this section we present some examples of K-finite objects and some counterexamples to certain properties.

\begin{example}
 
Let $\mathcal{S}$ be a model of ZF set theory with an infinite Dedekind finite set $A$.  See \cite{BlaSc89}.  If we force with finite partial monomorphisms (similarly to IV.5 above) we obtain a model of ZF in which the pullback of $A$ is Dedekind infinite.  (Forcing extensions admit a canonical geometric morphism to the base category of sets.)

\end{example}

\begin{example}
 
The powerset of a GDF object need not be GDF.  Let $Sets^{\mathbb{N}}$ be the topos of covariant functors on the natural numbers with their usual order.  The powerset of the terminal object can be pictured as in the diagram preceding A1.6.10 in \cite{Johns02}.  The pullback of $P(1)$ along the inclusion of any point $n:Sets\rightarrow Sets^{\mathbb{N}}$ is infinite, hence $P(1)$ is not GDF.  Thus the powerset of a GDF object need not be GDF.

\end{example}

\begin{example}
 
Let $P$ be the sheaf of germs of functions $g:\mathbb{R}\rightarrow\mathbb{N}$ in $Sh(\mathbb{R})$ such that $g(r)=0$ for all but (possibly) finitely many rational numbers and such that $g(p/q)\leq q$.  $P$ is evidently none of K, \~{K}, or R-finite (its stalks have arbitrarily large finite cardinality), but it is GDF.  See \cite{Freyd06}.

\end{example}

%% file: chap5.tex
\section{The Axioms F1 and F2}
\label{The Axioms F1 and F2}

Let $\mathcal{E}$ be a topos with a natural numbers object.  Throughout this section, $A$ will be a commutative, unital, non-trivial ring object in $\mathcal{E}$.  We will be interested in different versions of what it means for $A$ to be a field.  There are several possible axioms for fields which are intuitionistically inequivalent.  We shall be interested in the the following two:\bigskip

F1:  $\forall a(\top\rightarrow (a=0)\vee \exists b(ab=1))$\bigskip

F2:  $\forall a(a\neq 0\rightarrow \exists b(ab=1))$\bigskip

The axiom F1 is evidently geometric.  We will show F2 is not.  These axioms have been studied in \cite{Johns77(2)}.  The reason for considering the second axiom is that F1 implies the ring object $A$ is, in fact, decidable.  Following \cite{Johns77(2)} we give the name geometric field to those ring objects satisfying F1, and field to those satisfying F2.  We shall show that F2 plus the axioms of commutative, unital, non-trivial ring objects is not equivalent to a geometric theory, but it is equvalent to the degeneracy of the classifying topos of a certain geometric theory.\bigskip

In order to do this we will define an auxilliary theory NTI, and consider the theory T which says ``A is a commutative, unital, nontrivial ring and the classifying topos of NTI is degenerate'' which is evidently preserved by inverse images, and show that a ring object satisfies F2 if and only if it satisfies T.  We also note that the theory T+$\forall a(a=0\vee a\neq0)$ (that is, $A$ is decidable) will then be equivalent to F1, since Johnstone shows that F2+$\forall a(a=0\vee a\neq0)$ is equivalent to F1.\bigskip

\begin{definition}
NTI is the internal propositional theory whose sentence symbols are the members of $A$ and whose axioms are:\bigskip

1.)  $a\wedge b\rightarrow a+b$ \hfill $\forall a,b\in A$ \bigskip

2.)  $a\rightarrow ab$ \hfill $\forall a,b\in A$ \bigskip

3.)  $\top\rightarrow 0$\bigskip

4.)  $1\rightarrow \bot$\bigskip

5.)  $\top\rightarrow \bigvee_{a\neq 0}a$\bigskip

\end{definition}\bigskip

Intuitively, a model of NTI is a non-trivial ideal in $A$, that is, an ideal with a member distinct from $0$.  Like any geometric theory, NTI has a classifying topos $\mathcal{E}[$NTI$]$, which we describe explicitly.  The underlying category object of the site for $\mathcal{E}[$NTI$]$ is the internal poset $K(A)$ of K-finite subobjects of $A$, with the opposite of the inclusion order.  The axioms of NTI impose a Grothendieck topology on $K(A)$ as follows:\bigskip   

Axiom 1 tells us that for any $a,b\in A$ the set $\{a,b\}$ is covered by the sieve generated by $\{a,b,a+b\}$.  Axiom 2 tells us that for any $a,b\in A$ the set $\{a\}$ is covered by the sieve generated by $\{a,ab\}$.  Axiom 3 tells us that $\emptyset$ is covered by the sieve generated by $\{0\}$.  Axiom 4 tells us that $\{1\}$ is covered by the empty sieve.  Finally, Axiom 5 tells us that $\emptyset$ is covered by the sieve generated by the family $\{\{a\}|a\neq 0\}$.  A general cover is obtained from these basic covers by pullback and iteration.\bigskip

To obtain the pullback closure of these sieves, we make the following definition:  A subset $R\subseteq K(A)$ \textit{directly covers} an object $p\in K(A)$ if either:\bigskip

1.)  $\exists a,b\in p$ such that $p\cup \{a+b\}\in R$\bigskip

2.)  $\exists a\in p, \exists b\in A$ such that $p\cup \{ab\}\in R$\bigskip

3.)  $p\cup \{0\}\in R$\bigskip

4.)  $1\in p$\bigskip

5.)  $\forall a\neq 0(p\cup\{a\}\in R)$\bigskip

The notion of direct covering, when restricted to sieves, gives the pullback closure of the basic covers.  To get the closure under iteration, given a subset $R\subseteq K(A)$, we define $\bar{R}$ to be the smallest subset of $K(A)$ that contains $R$ and contains every object it directly covers.  Then the Grothendieck topology induced by the axioms of NTI is given by $J(p)=\{R\subseteq K(A)|R$ is a sieve on $p$ and $p\in \bar{R}\}$.\bigskip

For convenience we introduce a unary predicate $\in U$ denoting the extension of the formula $\phi(a)=\exists b(ab=1)$ so we may write $a\in U$ for $\phi(a)$.  We are now prepared to prove our main theorem.\bigskip

\begin{theorem}
The classifying topos of NTI is degenerate if and only $A$ satisfies the axiom F2.
\end{theorem}

\begin{proof}
 
$(\Rightarrow)$  Let $Z=\{p\in K(A)|$If $p=\emptyset$ then $\forall a(a\neq 0\rightarrow a\in U)$ and if $p=\{a_{0},...,a_{n}\}$ then $(\bigvee_{i=0}^{n}(a_{i}=1))\vee\bigwedge_{i=0}^{n}((a_{i}\in U) \vee \forall a(a\neq a_{i}\rightarrow a\in U))\}$.  We denote the condition for inhabited $p$ by $\phi$.  $Z$ is not empty since $\{1\}\in Z$.  We show that $Z$ contains everything it directly covers.  Suppose $Z$ directly covers $p$.  If $p=\emptyset$ then either:\bigskip

1.)  $\{0\}\in Z$, hence, by the definition of $Z$, since $0\notin U$ it must be the case that $\forall a(a\neq 0\rightarrow a\in U)$, so the first requirement is satisfied and $p=\emptyset\in Z$.\bigskip

2.)  $\forall b(b\neq 0\rightarrow \{b\}\in Z)$, hence, again by the definition of $Z$, since $\forall a\forall b(b\neq 0\wedge a\neq b\rightarrow a\in U)$ is false (for instance, with $a=0$), it must be the case that $b\in U$ and again the first requirement is satisfied and $p=\emptyset\in Z$.\bigskip

If $p=\{a_{0},...,a_{n}\}$ then either:\bigskip

1.)  For some $i,j \leq n$ $p\cup \{a_{i}+a_{j}\}\in Z$, hence the formula $\phi$ holds of $p\cup \{a_{i}+a_{j}\}$.  But this clearly implies that the formula $\phi$ also holds of $p$, so $p\in Z$.\bigskip

2.)  For some $i\leq n$ and some $b\in a$, $p\cup \{a_{i}b\}\in Z$, hence the formula $\phi$ holds of $p\cup \{a_{i}b\}\in Z$.  Again this implies that the formula $\phi$ also holds of $p$, so $p\in Z$.\bigskip

3.)  $p\cup \{0\}\in Z$, hence the formula $\phi$ holds of $p\cup \{0\}$.  Again this implies that the formula $\phi$ holds of $p$, so $p\in Z$.\bigskip

4.)  $1\in p$, hence, since $1\in U$ is true, the formula $\phi$ holds of $p$, so $p\in Z$.\bigskip

5.)  $\forall a(a\neq 0\rightarrow p\cup \{a\}\in Z)$, hence, since there is at least one element of $A$ not equal to $0$ (namely $1$), and the formula $\phi$ holds for all $p\cup \{a\}$, the formula $\phi$ again holds for $p$, so $p\in Z$.\bigskip

If the classifying topos of NTI is degenerate, then every set which contains everything it directly covers must contain $\emptyset$, hence $\emptyset\in Z$ and therefore $\forall a(a\neq 0\rightarrow a\in U)$, i.e. $A$ satisfies F2.\bigskip

($\Leftarrow$)  Suppose $Z$ contains everything it directly covers.  If $1\in p$ then $p\in Z$, hence if $a\neq 0$ then, since $\exists b(ab=1)$, $\{a\}\cup \{ab\}=\{a\}\cup \{1\}\in Z$.  But then $\{a\}\in Z$, so $Z$ contains every non-zero singleton.  But then $Z$ directly covers $\emptyset$, so $\emptyset\in Z$.  Since $Z$ was arbitrary, $\emptyset\in \bar{\emptyset}$, hence the classifying topos of NTI is degenerate.

\end{proof}

\section{Examples}

\begin{example}
 
Consider the Sierpinski topos $Sets^{2}$ whose objects are triples $(A,B,g)$ where $A$ and $B$ are sets and $g:A\rightarrow B$ is a function.  If an object $(A,B,g)$ satisfies F2, then $B$ must be a field in $Sets$, that is, satisfy F1, $A$ must be a ring object, and $g$ must be a homomorphism.  If $(a,b)\neq (0,0)$ then $a\neq 0$ and $b\neq 0$, hence if $(A,B,g)$ satisfies F2, $a$ and $b$ are invertible.  It follows that $A$ must be a local ring with maximal ideal the kernel of $g$, and $B$ is an extension of the residue field of $A$.  For instance, given any field $k$, $k[x]/(x^{2})\rightarrow k$, where $x\mapsto 0$, satisfies F2.  But the pullback along the closed point of $Sets^{2}$ is $k[x]/(x^{2})$ which is evidently not a field.  Thus F2, and therefore NTI are not preserved by geometric morphisms.  See \cite{Johns77(2)}.  

\end{example}

How can it be that NTI is not preserved by geometric morphisms?  The answer lies in the fifth axiom.  It is true that the classifying topos of NTI($A$) is degenerate, then so is the classifying topos of $f^{*}($NTI$(A))$.  But since $f^{*}(\{a\in A|\neg(a=0)\})$ is, in general, different from $\{a\in f^{*}(A)|\neg(a=0)\}$, the canonical translation of the fifth axiom need not be the fifth axiom applied to $f^{*}(A)$, and therefore the canonical translation of NTI($A$) is, in general, different from NTI($f^{*}(A)$). Thus we see that properties equivalent to the degeneracy of the classifying topos for a geometric theory are only preserved by geometric morphisms when the joins considered in the axioms are preserved by inverse images.

%% file: thesis.bbl
\begin{thebibliography}{10}

\bibitem{Ortega92}
O.~Acu{\~{n}}a-Ortega.
\newblock Finite objects and automorphisms.
\newblock {\em Communications in Algebra}, 20:3459--3478, 1992.

\bibitem{ArGrV63}
M.~Artin, A.~Grothendieck, and J.L. Verdier.
\newblock {\em {S}\'{e}minaire de {G}\'{e}om\'{e}trie {A}lg\'{e}brique du
  {B}ois {M}arie - 1963-64 - {T}h\'{e}orie des topos et cohomologie \'{e}tale
  des sch\'{e}mas - ({SGA} 4)}, volume 269 of {\em Lecture Notes in
  Mathematics}.
\newblock Springer, Berlin, 1972.

\bibitem{Bell88}
J.L. Bell.
\newblock {\em Toposes and local set theories: an introduction}.
\newblock Clarendon Press Oxford, 1988.

\bibitem{Blass87}
A.~Blass.
\newblock Well ordering and induction in intuitionistic logic and topoi.
\newblock In D.W. Kueker, E.G.K. Lopez-Escobar, and C.H. Smith, editors, {\em
  Mathematical Logic and Theoretical Computer Science}, volume 106 of {\em
  Lecture Notes in Pure and Applied Mathematics}, pages 29--48. Marcel Dekker,
  1987.

\bibitem{BlaSc83}
A.~Blass and A.~Scedrov.
\newblock Classifying topoi and finite forcing.
\newblock {\em Journal of Pure and Applied Algebra}, 28(2):111--140, 1983.

\bibitem{BlaSc89}
A.~Blass and A.~Scedrov.
\newblock {\em Freyd's models for the independence of the axiom of choice}.
\newblock Number 404. Memoirs of the American Mathematical Society, 1989.

\bibitem{Caram11}
O.~Caramello.
\newblock A characterization theorem for geometric logic.
\newblock {\em Annals of Pure and Applied Logic}, 162(4):318--321, 2011.

\bibitem{Diaco75}
R.~Diaconescu.
\newblock Change of base for toposes with generators.
\newblock {\em Journal of Pure and Applied Algebra}, 6:191--218, 1975.

\bibitem{Freyd06}
P.~Freyd.
\newblock Numerology in topoi.
\newblock {\em Theory and Applications of Categories}, 16(19):522--528, 2006.

\bibitem{Johns77(2)}
P.T. Johnstone.
\newblock Rings, fields, and spectra.
\newblock {\em Journal of Algebra}, 49(1):238--260, 1977.

\bibitem{Johns77}
P.T. Johnstone.
\newblock {\em Topos theory}.
\newblock Academic Press, 1977.

\bibitem{Johns02}
P.T. Johnstone.
\newblock {\em Sketches of an elephant: a topos theory compendium}.
\newblock Oxford University Press, USA, 2002.

\bibitem{Stout87}
L.N. Stout.
\newblock Dedekind finiteness in topoi.
\newblock {\em Journal of Pure and Applied Algebra}, 49, 1987.

\end{thebibliography}
